\documentclass[11pt,reqno]{amsart}
\usepackage[margin=1.05in]{geometry}
\usepackage{amsmath,amssymb,amsthm,mathtools,booktabs,array,enumitem,microtype, amsaddr}
\usepackage{xcolor}
\usepackage[section]{placeins}
\usepackage{tikz}
\usetikzlibrary{arrows.meta,positioning,calc}
\definecolor{figblue}{HTML}{24567A}
\definecolor{figteal}{HTML}{16756E}
\definecolor{figorange}{HTML}{AB6125}
\usepackage[colorlinks=true,linkcolor=blue!50!black,citecolor=blue!50!black,urlcolor=blue!50!black]{hyperref}
\hypersetup{pdftitle={Rank-sensitive vertex bounds for semidefinite lifts},pdfsubject={Rank-sensitive vertex counting, sums of squares, and polygon geometry},pdfauthor={Avinash Bhardwaj}}
\theoremstyle{plain}
\newtheorem{theorem}{Theorem}[section]
\newtheorem{proposition}[theorem]{Proposition}
\newtheorem{lemma}[theorem]{Lemma}
\newtheorem{corollary}[theorem]{Corollary}

\theoremstyle{definition}

\newtheorem{remark}[theorem]{Remark}
\DeclareMathOperator{\rank}{rank}
\DeclareMathOperator{\conv}{conv}
\DeclareMathOperator{\tr}{tr}

\newcommand{\R}{\mathbb R}
\newcommand{\PSD}[1]{\mathbb S^{#1}_+}
\newcommand{\xcLP}{\operatorname{xc}_{\mathrm{LP}}}
\newcommand{\xcPSD}{\operatorname{xc}_{\mathrm{PSD}}}
\newcommand{\rkplus}{\operatorname{rank}_+}
\newcommand{\rkpsd}{\operatorname{rank}_{\mathrm{psd}}}
\newcommand{\rksqrt}{\operatorname{rank}_{\sqrt{\phantom{x}}}}

\newcommand{\Par}{\mathrm{PAR}}

\newcommand{\sosc}{\operatorname{sosdim}}

\begin{document}
\title[Rank-sensitive vertex bounds for semidefinite lifts]{Rank-sensitive vertex bounds for semidefinite lifts}
\author{Avinash Bhardwaj}
\email{abhardwaj@iitb.ac.in}
\address{Department of Industrial Engineering and Operations Research,\\ Indian Institute of Technology Bombay,\\ Mumbai, India 400076}
\date{October 2026}
\begin{abstract}
We study how the ranks of factors in a positive semidefinite slack factorization constrain its size. We bound the number of low-rank vertex factors in any fixed factorization, and use this to determine the asymptotic order of the minimum factorization size under any fixed bound on vertex-factor ranks: $\Theta(n\log n)$ for permutahedra $\Pi_n$ and $\Theta(\log N)$ for regular $N$-gons. In the rank-one
case this gives the order of the minimum dimension of a space of real functions on the vertices in which every facet slack is a sum of squares, with no symmetry or degree restriction. For unrestricted lifts, the bound gives $\mathrm{xc}_{\mathrm{PSD}}(\Pi_n)\ge n+\log_3 n-O(1)$, and shows that a permutahedron factorization of linear size would need vertex factors of rank $\Omega(\log n)$ at all but a vanishing fraction of vertices.

We also prove that every polytope of real positive semidefinite rank at most four has at most twelve vertices, and construct explicit size-four lifts for every square-symmetric octagon, so eight is attained. The upper bound combines a count of the polygon corners reached by curves of rank-one factors with incidence constraints on factor ranks. Whether eight is the maximum remains open.
\end{abstract}
\maketitle

\section{Introduction}\label{sec:intro}
With any fixed bound on the ranks of the vertex factors, the minimum
size of a real PSD slack factorization is $\Theta(n\log n)$ for the
permutahedron $\Pi_n$ and $\Theta(\log N)$ for the regular $N$-gon $R_N$.
The rank-one case determines the minimum dimension of an SOS function
space on the vertices, allowing arbitrary real functions. Neither
symmetry nor a polynomial degree bound is imposed. We obtain these
lower bounds from one counting theorem, which controls the number of
low-rank vertex factors in any fixed factorization. It also quantifies
how frequently higher ranks must occur in smaller lifts.

A semidefinite lift represents a polytope as an affine image of an affine
slice of the cone of positive semidefinite matrices. Its size is the
matrix order. Let $\xcPSD(P)$ denote the least such size for a polytope
$P$, and let $\xcLP(P)$ be the least number of inequalities in a linear
lift. The basic bounds are
\begin{equation}\label{eq:comparison}
 \dim P+1\le\xcPSD(P)\le\xcLP(P).
\end{equation}
For facet slacks $\ell_i$ and vertices $v$, a size-$r$ PSD slack
factorization consists of real PSD matrices satisfying
$\ell_i(v)=\tr(A_iB_v)$. The minimum size is $\xcPSD(P)$
\cite{GPT2013}. Write $r_k(P)$ for the minimum size when every vertex
factor has rank at most $k$, with unrestricted facet factors. This asks
how much of the advantage of semidefinite lifts requires vertex factors
of growing rank. Our main result counts low-rank vertex factors within any fixed factorization,
even if other vertex factors have higher rank.

\begin{theorem}\label{thm:count}
Let $P$ be a $d$-dimensional polytope with a fixed size-$r$ real PSD
slack factorization $\ell_i(v)=\tr(A_iB_v)$. For $1\le k\le r$, put
\[
 N_{\le k}=\bigl|\{v\in\operatorname{vert}(P):\rank B_v\le k\}\bigr|.
\]
Then
\begin{equation}\label{eq:low-rank-count}
 N_{\le k}\le\prod_{i=r-k+1}^{r}\frac{3^i+1}{2}-1
 \le\frac{3^{rk-\binom{k}{2}}-1}{2}.
\end{equation}
No symmetry, inscribability, or genericity assumption is imposed on $P$
or its factorization.
\end{theorem}

Write $\sosc(P)$ for the minimum dimension of a space of real functions
on the vertices in which every facet slack is a sum of squares. These
functions need not be polynomial, and constants need not belong to the
space. Equivalently, $\sosc(P)$ is the minimum PSD factorization size
when all vertex factors have rank one, so $\sosc(P)=r_1(P)$.
In any size-$r$ factorization
of a $d$-polytope, every vertex factor has rank at most $r-d$.
These standard factorization facts are recalled in Section~\ref{sec:prelim}.
Taking $k=1$ gives $\sosc(P)\ge\lceil\log_3(2N+1)\rceil$ for a
polytope with $N$ vertices. Taking $k=r-d$ gives a bound for unrestricted
lifts that retains the dimension of the polytope. The exact forms and
the comparison with existing vertex bounds are given in
Section~\ref{sec:general-consequences}.

This is a semidefinite counterpart of Goemans' face-counting argument
\cite{Goemans2015}. A linear extension with $m$ inequalities has at most
$2^m$ faces, and each vertex of its image lifts to a distinct face, so
$m\ge\log_2 N$. A spectrahedral slice can have infinitely many faces.
Here a rank restriction instead bounds the number of vertex factors
in a fixed factorization, giving $r_k(P)=\Omega_k(\log N)$ for fixed $k$.

\subsection{Consequences}
For every fixed $k\ge1$, we obtain
\[
 r_k(\Pi_n)=\Theta(n\log n),\qquad r_k(R_N)=\Theta(\log N).
\]
The constants may depend on $k$. Corollary~\ref{cor:bounded-vertex-ranks}
proves these estimates by combining the count with linear extension
upper bounds. The lower bounds use only the vertex counts $n!$ and $N$,
not further structure of these families. In particular, allowing any fixed vertex-factor rank
does not improve the asymptotic order over rank one.

The count also constrains the structure of hypothetical smaller lifts:
in every linear-size permutahedron factorization, all but a vanishing
fraction of the vertex factors have rank $\Omega(\log n)$.
Corollary~\ref{cor:rank-distribution} quantifies this statement and gives
the corresponding conclusion for regular polygons: a factorization
of size $o(\log N)$ has a vanishing fraction of vertex factors of rank
at most any fixed constant.

For unrestricted lifts the count gives
\[
 \xcPSD(\Pi_n)\ge n+\log_3 n-O(1).
\]
Here $\Pi_n$ has dimension $n-1$ and $n!$ vertices. The explicit bound
first reaches $n+2,n+3,n+4$ at $n=22,72,221$, respectively.

The counting theorem controls the distribution of factor ranks in
arbitrary matrix orders and determines the asymptotic cost of keeping
vertex-factor ranks bounded. At matrix order four, additional incidence
and curve geometry gives a much sharper vertex bound: twelve, compared
with $573$ from the general count.
Let $V(r)$ be the maximum number of vertices among polytopes with PSD
rank at most $r$, and let $V_2(r)$ be the corresponding maximum among
polygons. These maxima are finite by \cite[Theorem 4]{FawziSafey2018}.

\begin{theorem}\label{cor:rank-four}
One has
\[
 8\le V_2(4)=V(4)\le12.
\]
\end{theorem}

The upper bound requires no symmetry or genericity and no finiteness
assumption on the rank-one loci of a factorization. Its proof combines
corner counts on those loci with the incidence restrictions on factor
ranks. The last configuration is excluded by expressing its factors as
cubic polynomials and comparing the multiplicities of their zeros. Section~\ref{sec:rank-four-bound} gives the complete argument;
Section~\ref{sec:octagons} gives a family attaining eight vertices.

\subsection{Proof method and relation to prior work}
The proof starts with the support function. On an open set of objective
directions $c$, an exposed vertex $v$ gives the support value $c^{\mathsf T}v$.
If these values satisfy one polynomial equation, each represented vertex
contributes a distinct linear factor. Its degree in the value variable
therefore bounds the number of vertices. Section~\ref{sec:proof} first
records this principle in a form that allows limits.

To obtain the polynomial equation, write a rank-at-most-$k$ PSD matrix
in triangular Gram coordinates, using $M=rk-\binom{k}{2}$ entries.
Normalization puts the relevant factors on a sphere, and a quartic
penalty enforces the remaining equations. After a small perturbation,
the sphere critical points rescale to solutions of $\nabla p(x)=x$.
B\'ezout bounds these cubic solutions, while independent changes of
column signs identify solutions with the same value. Burnside's lemma
accounts for the zero columns and gives the product in
\eqref{eq:low-rank-count}. A resultant records the critical values in one
polynomial; continuity carries its relation with the maximum through
exceptional parameters and to the limiting support function.
The rescaling is the standard tensor-eigenvector
device discussed in \cite[Remark 2.1]{CartwrightSturmfels2013}; we give
the needed argument directly.

The elimination and linear-factor strategy is closely related to
Fawzi--Safey El Din's argument \cite{FawziSafey2018}. Here the use of
low-rank Gram coordinates and a quartic penalty preserving column signs yields the
rank-sensitive exponent and the dimension-dependent consequence.
Their displayed bound is $2^{r^2}$ vertices for a convex body with a
size-$r$ semidefinite representation \cite[Theorem 4]{FawziSafey2018}.
Our dimension-dependent consequence improves that estimate for polytopes.
Their method also admits a refinement using exact SDP algebraic degrees;
Section~\ref{sec:general-consequences} distinguishes that refinement from
the displayed bound.

Other approaches address unrestricted PSD size. Lee--Raghavendra--Steurer
\cite{LeeRaghavendraSteurer2015} prove strong lower bounds for explicit
families, while Bri\"et--Dadush--Pokutta \cite{BrietDadushPokutta2015}
use counting to establish the existence of $0/1$ polytopes with large
PSD extension complexity. Their argument counts polytopes admitting
small lifts; ours counts low-rank vertex factors inside one fixed
factorization. Lee--Wei--de Wolf \cite{LeeWeiDeWolf2017} develop
entry-dependent PSD-rank bounds and connections with quantum
communication complexity.

The SOS interpretation of PSD factorizations is established theory
\cite{FSP2015equivariant,FGPST2020}; in particular,
\cite[Proposition 4.3]{FGPST2020} constructs a lift with rank-one
vertex factors from an SOS function space. The lower bound here applies without
restricting the functions to an equivariant or degree-bounded space.

\subsection{Polygon geometry and organization}
The rank-four polygon question provides a concrete test of how factor
ranks constrain lifts, using a second, geometric method. Proposition~\ref{prop:p3-spans} reduces
an arbitrary candidate, up to polarity, to a compact lift of affine
dimension at most five. Section~\ref{sec:rank-four-corners} explains how
curves can meet polygon corners, identifies the first possible dependence
among rank-one corner matrices, and uses the two endpoints of a facet to
restrict its signed representations. The key is to count polygon corners
attained by the rank-one locus. A curve may contain infinitely many
rank-one matrices, but Lemma~\ref{lem:corner-curve} bounds the number of
corners its image can reach. Section~\ref{sec:rank-four-bound} combines
these principles to prove the twelve-vertex bound without assuming
that the rank-one loci are finite. In the only possible thirteen-vertex
configuration, the six rank-one facet factors lie on a cubic moment
curve. Three consecutive vertex cubics force every element of their
space of squares that vanishes at one point to have a double zero
there. Another consecutive pair produces an element with a simple zero.
Section~\ref{sec:octagons} constructs size-four lifts for all
square-symmetric octagons. Their rank-one facet factors realize the
eight-term dependence, and their edge lines form the rank-three branch
of the explicit pencil. For the regular octagon, the minimum size rises
from four to five when the lift is required to respect rotations, in
the equivariant model of \cite{FSP2015equivariant,FSP2017polygons};
Corollary~\ref{cor:equivariant-octagon} gives a proof specialized to
this case. Whether eight is the maximum number of vertices at PSD rank four remains open.

Section~\ref{sec:prelim} records the factorization facts,
Section~\ref{sec:proof} proves the counting theorem, and
Section~\ref{sec:applications} derives the asymptotic orders, the
unrestricted bounds, and the rank-distribution consequences.
Section~\ref{sec:open} states the remaining questions and explains
the limits of support-based and symmetry-restricted methods.

\section{Factorizations and basic bounds}\label{sec:prelim}
All polytopes are nonempty and of positive dimension, and dimensions are
taken in their affine hulls. For point faces, we use the convention
$\xcPSD(\{v\})=1$. We write $\mathbb S^r$ for the real symmetric matrices and $\PSD r$ for their positive semidefinite cone. The size of
a lift is the order of one PSD matrix, not the number of its entries. A
product of PSD cones can be embedded as diagonal blocks in a single cone,
whose order is the sum of the block orders. Throughout, $\log$ means $\log_2$.
A cone lift expresses $P$ as an affine image of the intersection of a closed
convex cone with an affine subspace. We write $\xcLP(P)$ for the least $k$
admitting a lift over $\R^k_+$ and $\xcPSD(P)$ for the least $r$ admitting a
lift over $\PSD r$. Equations in either description are not counted.

Let $P=\conv\{v_1,\ldots,v_N\}$ have facet inequalities
$\ell_i(x)\ge0$. Its slack matrix is $S_P=(\ell_i(v_j))_{i,j}$.
The nonnegative rank of $M$ is the least inner dimension of a factorization
$M=UV$ with $U,V\ge0$. Its PSD rank is the least $r$ for which
\[
M_{ij}=\tr(A_iB_j),\qquad A_i,B_j\in\PSD r.
\]
The factorization theorems of Yannakakis and Gouveia--Parrilo--Thomas identify
\begin{equation}\label{eq:factorization}
\xcLP(P)=\rkplus(S_P),\qquad \xcPSD(P)=\rkpsd(S_P)
\end{equation}
\cite{Yannakakis1991,GPT2013}. These quantities are invariant under affine
equivalence and positive row or column scalings of a slack matrix.

\begin{lemma}\label{lem:basic}
For a $d$-dimensional polytope $P$,
\[
\xcPSD(P)\ge d+1=\rank S_P.
\]
If $M$ is nonnegative and $R\circ R=M$, then $\rkpsd(M)\le\rank R$.
\end{lemma}
\begin{proof}
The first statement is \cite[Proposition 3.2]{GRT2013}. For the second,
write $R_{ij}=a_i^{\mathsf T}b_j$ with $a_i,b_j\in\R^r$ and observe that
$M_{ij}=\tr(a_ia_i^{\mathsf T}b_jb_j^{\mathsf T})$.
\end{proof}
The least possible rank of such an $R$ is the square-root rank $\rksqrt(M)$.
The signs of its entries may be chosen independently. For general
nonnegative matrices, ordinary rank is not a lower bound for PSD rank;
one only has $\rank M\le\binom{\rkpsd(M)+1}{2}$ \cite{FGPRT2015}.

\begin{lemma}\label{lem:faces}
If $F$ is a nonempty face of $P$, then
\begin{equation}\label{eq:face-bound}
\xcPSD(P)\ge\xcPSD(F)+\dim P-\dim F.
\end{equation}
Affine images and nonempty affine sections cannot increase PSD extension
complexity. If $0\in\operatorname{int}P$, then $\xcPSD(P^\circ)=\xcPSD(P)$.
\end{lemma}
\begin{proof}
For a facet $H$, take a size-$r$ PSD factorization of $S_P$, with facet factor
$A_H\ne0$. If $v\in H$, positivity and $\tr(A_HB_v)=0$ imply
$\operatorname{im}B_v\subseteq\ker A_H$. Compress the factors to this kernel.
The facet inequalities of $H$ are restrictions of suitable facet inequalities
of $P$, so this gives a factorization of $S_H$ of size at most $r-1$.
Iteration along a face flag proves \eqref{eq:face-bound}; this is the argument
of \cite[Proposition 3.8]{GRT2013}. An affine section is imposed by extra
affine equations on the lift, and an image by composing the projection.
Polarity transposes a slack matrix up to positive scaling.
\end{proof}
An intersection with a halfspace is not an affine section. In particular,
arbitrarily deleting or truncating polygon vertices is not an operation under
which PSD rank is known to be nonincreasing.

\begin{proposition}\label{prop:rank-budget}
In a size-$r$ real PSD factorization of a $d$-polytope, every facet factor
and every vertex factor has rank at most $r-d$. If $d\ge2$ and $v$ lies
on a facet $F$, then
\[
\rank A_F+\rank B_v\le r-d+1.
\]
In particular, when $r=d+2$, every factor has rank one or two, and an
incident facet and vertex cannot both have rank-two factors.
\end{proposition}
\begin{proof}
Write $a=\rank A_F$. Compression to $\ker A_F$, as in the proof of
Lemma~\ref{lem:faces}, factors the slack matrix of $F$ in size $r-a$.
For $d\ge2$, the dimension bound gives $r-a\ge d$. For a segment
($d=1$), the same inequality follows because the factor of its endpoint
is nonzero and supported on this kernel. Transposing a slack matrix of
the polar gives the identical rank bound for vertex factors.
Now let $d\ge2$ and $v\in F$. In the compressed factorization of $F$,
the rank of $B_v$ is unchanged. Applying the vertex-factor bound just
proved in dimension $d-1$ gives
$\rank B_v\le(r-a)-(d-1)$, as required. Nonzero slack rows and columns
ensure that all factors are nonzero. This refines the usual kernel
argument for the dimension bound \cite{GRT2013}.
\end{proof}

\subsection{Triangular Gram factors and SOS dimension}
\begin{lemma}\label{lem:gram}
Every $B\in\PSD r$ of rank at most $k\le r$ can be written
$B=YY^{\mathsf T}$, where $Y\in\R^{r\times k}$ and $Y_{ij}=0$ for
$i<j$. Such matrices $Y$ have
\begin{equation}\label{eq:gram-dimension}
 M(r,k)=\sum_{j=1}^k(r-j+1)=rk-\binom{k}{2}
\end{equation}
free entries, and $\tr B=\|Y\|_F^2$.
\end{lemma}
\begin{proof}
Start with any Gram factor $Y_0$, padding with zero columns, and take
an orthogonal QR decomposition $Y_0^{\mathsf T}=QR$. Then
$Y=R^{\mathsf T}$ has the prescribed zeros and
$YY^{\mathsf T}=Y_0Y_0^{\mathsf T}$. This works at every rank and
requires no common choice of nonzero pivots.
\end{proof}

\begin{proposition}\label{prop:sos-rank}
The quantity $\sosc(P)$ equals the minimum factorization size
when every vertex factor has rank one. Moreover, if $d=\dim P$ and
$r=\xcPSD(P)$, then
\begin{equation}\label{eq:sos-gap}
 r\le\sosc(P)\le
 \min\left\{\xcLP(P),\binom{r+1}{2}-\binom{d+1}{2}\right\}.
\end{equation}
\end{proposition}
\begin{proof}
If $f_1,\ldots,f_s$ is a basis of an admissible function space, set
$f(v)=(f_1(v),\ldots,f_s(v))$. The SOS identities have the form
$\ell_i(v)=f(v)^{\mathsf T}A_if(v)$ with $A_i\succeq0$, giving
$B_v=f(v)f(v)^{\mathsf T}$. These vectors are nonzero, since no vertex
lies on every facet. Conversely, the coordinates of $b_v$ in a
factorization with $B_v=b_vb_v^{\mathsf T}$ span an admissible space.
This proves the equivalence and the lower bound.

For the LP upper bound, a nonnegative factorization $S=TU$ of inner
dimension $q$ gives functions $f_a(v)=\sqrt{U_{av}}$, since
$S_{iv}=\sum_aT_{ia}f_a(v)^2$. For the other upper bound, apply
Proposition~\ref{prop:rank-budget} and Lemma~\ref{lem:gram} with
$k=r-d$. The $M(r,k)$ free entries of one chosen Gram factor $Y_v$,
viewed as functions of $v$, certify each facet slack because
$\ell_i(v)=\|A_i^{1/2}Y_v\|_F^2$.
\end{proof}
The equivalence is the standard function-space SOS interpretation
\cite{FSP2015equivariant,FGPST2020}; the triangular count refines the
usual conversion bound. Adjoining constants increases dimension by
at most one. We use \emph{SOS dimension} as descriptive terminology
for this one-sided minimum. The survey \cite{FGPRT2015} gives no
separate name for it; its square-root rank instead requires rank-one
factors on both sides \cite[Definition 5.2 and Proposition 6.2]{FGPRT2015}.

\section{Proof of the universal vertex count}\label{sec:proof}
Let $X$ be a compact subset of a polytope. If $X$ contains a vertex $v$,
then its support function agrees with $c^{\mathsf T}v$ throughout the
interior of the normal cone at $v$. Thus vertices give distinct linear
branches of the support function. We first explain how a polynomial
relation counts these branches, even when obtained by a limit. We then
construct such a relation by counting critical values of quartics.

\subsection{Vertex counts in the limit}
\begin{lemma}\label{lem:linear-branches}
Let $X$ be a nonempty compact subset of a polytope $P\subseteq\R^d$,
and put $h_X(c)=\max_{x\in X}c^{\mathsf T}x$.
Suppose real functions $h_\varepsilon(c)$ converge to $h_X(c)$ for
every fixed $c$ as $\varepsilon\downarrow0$. If a nonzero real polynomial
$F(\varepsilon,c,z)$ satisfies
\[
 F(\varepsilon,c,h_\varepsilon(c))=0
 \qquad(\varepsilon>0,\ c\in\R^d),
\]
then at most $\deg_z F$ vertices of $P$ belong to $X$.
\end{lemma}
\begin{proof}
Cancel the largest power of $\varepsilon$ dividing $F$. The relation
is preserved for $\varepsilon>0$, and now
\[
 F_0(c,z):=F(0,c,z)\ne0,\qquad \deg_z F_0\le\deg_z F.
\]
Passing to the limit gives $F_0(c,h_X(c))=0$ for every $c$.

Fix a vertex $v\in X$. On the open interior of its normal cone in $P$,
$h_X(c)=c^{\mathsf T}v$. Thus the polynomial
$F_0(c,c^{\mathsf T}v)$ vanishes on a real open set and is identically
zero. Hence $z-c^{\mathsf T}v$ divides $F_0$. Distinct vertices give
distinct monic linear factors in $z$, proving the bound.
\end{proof}

Only the degree in $z$ matters. Neither convergence of individual
critical points nor finiteness of the limiting critical-point set is
required.

\subsection{Critical values of a quartic penalty}
The next lemma supplies the polynomial relation. Independent signs of
the Gram columns will become independent signs of the blocks below.

\begin{lemma}\label{lem:corners}
Partition $y\in\mathbb R^M$ into blocks $y_j\in\mathbb R^{m_j}$, where $m_j\ge1$ and $\sum_{j=1}^k m_j=M$. Let $q$ have degree at most four, be nonnegative on the unit sphere, and be invariant under changing the sign of any entire block independently. Suppose
\[
 T=\{y:\|y\|^2=1,\ q(y)=0\}
\]
is nonempty. Let $\phi:\mathbb R^M\to\mathbb R^d$ have coordinate polynomials of degree at most two with the same invariance, and suppose $\phi(T)\subseteq P$ for a polytope $P$. The number of vertices of $P$ belonging to $\phi(T)$ is at most
\[
 C(m_1,\ldots,m_k):=\prod_{j=1}^k\frac{3^{m_j}+1}{2}-1.
\]
\end{lemma}
\begin{proof}
\emph{Approximating the support function.}
We will apply Lemma~\ref{lem:linear-branches} to $X=\phi(T)$.
Put $\rho(y)=y^{\mathsf T}y$. Block-sign invariance implies that every
homogeneous term has even degree. Replacing $q_4+q_2+q_0$ by
$q_4+\rho q_2+\rho^2q_0$, and each coordinate $\phi_2+\phi_0$ by
$\phi_2+\rho\phi_0$, changes neither polynomial on the unit sphere.
We may therefore assume that $q$ is homogeneous quartic and $\phi$
homogeneous quadratic.

Use the fixed quartic $R(y)=\sum_{i=1}^M y_i^4$, which preserves all
block signs. For $\varepsilon>0$, set
\[
 p_{\varepsilon,c}(y)=-q(y)+\varepsilon\rho(y)c^{\mathsf T}\phi(y)
 +\varepsilon^2R(y),\qquad
 f_{\varepsilon,c}=p_{\varepsilon,c}/\varepsilon.
\]
On the unit sphere, $f=c^{\mathsf T}\phi-q/\varepsilon+\varepsilon R$.
The penalty forces maximizers toward $T$, while the smaller perturbation
will make the critical-point equations finite.
Its maximum $h_\varepsilon(c)$ is continuous in $(\varepsilon,c)$ for
$\varepsilon>0$, since the sphere is compact. For each fixed $c$,
$h_\varepsilon(c)\to\max_T c^{\mathsf T}\phi$. Indeed, evaluating at a
maximizer on $T$ gives the lower limit inequality. Boundedness on the
sphere and comparison with a point of $T$ give
$q(y_\varepsilon)=O(\varepsilon)$ at any maximizer $y_\varepsilon$.
Every accumulation point therefore lies in $T$, and $-q/\varepsilon\le0$
gives the upper limit inequality.

\emph{Counting critical points.}
A maximum on the sphere is a critical value. Homogeneity turns the
critical-point equations into a cubic system with no multiplier.
For generic parameters with $\varepsilon\ne0$, the homogeneous quartic
$p=p_{\varepsilon,c}$ satisfies
\begin{equation}\label{eq:smooth-quartic}
 \nabla p(x)=0\quad\Longrightarrow\quad x=0
 \qquad (x\in\mathbb C^M).
\end{equation}
This condition is the nonvanishing of the resultant of the partial
derivatives. For each fixed $c$, that resultant is a nonzero polynomial
in $\varepsilon$: as $\varepsilon\to\infty$, the coefficients of
$p_{\varepsilon,c}/\varepsilon^2$ tend to those of $R$, whose gradient
vanishes only at the origin. Thus only finitely many $\varepsilon$ are
excluded for each $c$.

Consider the cubic system
\begin{equation}\label{eq:kkt}
 \nabla p(x)=x.
\end{equation}
Its homogenization is $\partial_i p(x)=t^2x_i$ in $\mathbb P^M$.
Condition \eqref{eq:smooth-quartic} excludes solutions at $t=0$.
The projective zero set is consequently finite: a positive-dimensional
projective variety meets every hyperplane. B\'ezout's theorem bounds its
length by $3^M$. The origin is a simple affine solution, since the
Jacobian of $\nabla p(x)-x$ there is $-I$. Thus
\begin{equation}\label{eq:degree}
 \bigl|\{x\ne0:\nabla p(x)=x\}\bigr|\le3^M-1.
\end{equation}

Let $K$ be the subset of these solutions with $\rho(x)\ne0$. Every
critical value of $f$ on the complex sphere is represented on $K$.
Indeed, at a sphere critical point $y$ we have $\nabla p(y)=\lambda y$.
Condition \eqref{eq:smooth-quartic} gives $\lambda\ne0$, and
$x=y/\sqrt\lambda$ satisfies \eqref{eq:kkt}, with
$\rho(x)=1/\lambda$. Euler's identity gives $4p(y)=\lambda$, so
\begin{equation}\label{eq:cubic-value}
 f(y)=\frac{1}{4\varepsilon\rho(x)}.
\end{equation}
Thus isotropic solutions need not be counted, and no zero-eigenvalue
exception occurs.

\emph{Identifying equal values by symmetry.}
Changing block signs preserves the value in \eqref{eq:cubic-value}.
We therefore count the orbits of $G=(\mathbb Z/2\mathbb Z)^k$ on $K$.
For $S\subseteq[k]$, write
$M_S=\sum_{j\in S}m_j$. The fixed points of the element that leaves
precisely the blocks in $S$ unflipped solve the restricted cubic system
on those blocks. The omitted gradient coordinates vanish by block-sign
invariance. For the same reason, the restricted quartic satisfies
\eqref{eq:smooth-quartic}; otherwise its nonzero gradient zero would
also be a gradient zero of $p$. The preceding B\'ezout argument gives
$|\operatorname{Fix}_K(g)|\le3^{M_S}-1$, including zero for
$S=\varnothing$. Burnside's lemma now yields
\begin{equation}\label{eq:sign-orbits}
 |K/G|=2^{-k}\sum_{g\in G}|\operatorname{Fix}_K(g)|
 \le2^{-k}\sum_{S\subseteq[k]}(3^{M_S}-1)
 =C(m_1,\ldots,m_k).
\end{equation}
This includes the stabilizers at zero blocks, which prevent simply
dividing the full count by $2^k$.

\emph{Passing from critical values to the limit.}
Form the homogeneous resultant $\mathcal R(\varepsilon,c,z)$ of
\[
 \partial_i p(x)-t^2x_i=0\quad(1\le i\le M),\qquad
 4\varepsilon z\rho(x)-t^2=0.
\]
It is a polynomial in the parameters, vanishing exactly when these
equations have a common zero in $\mathbb P^M$
\cite[Chapter 3]{CoxLittleOShea2005}.
Under \eqref{eq:smooth-quartic}, no such zero has $t=0$. For
$\varepsilon\ne0$, its roots in $z$ are therefore precisely the values
in \eqref{eq:cubic-value}. In particular, $\mathcal R$ is nonzero:
$z=0$ would force $t=0$.

Factor in $\R[\varepsilon,c,z]$ as
\[
 \mathcal R=A(\varepsilon,c)\prod_i F_i(\varepsilon,c,z)^{e_i},
 \qquad F=\prod_i F_i,
\]
where the $F_i$ are distinct irreducible factors of positive degree in
$z$. For generic parameters, $F$ retains its degree in $z$ and has
distinct roots, since the coefficient field has characteristic zero.
Each root is a critical value, so \eqref{eq:sign-orbits} gives
$\deg_z F\le C(m_1,\ldots,m_k)$.

For generic real parameters with $\varepsilon>0$, the maximum
$h_\varepsilon(c)$ is one of these values, and hence
$F(\varepsilon,c,h_\varepsilon(c))=0$. Both the maximum and the
polynomial are continuous. The same identity therefore holds for
\emph{every} $\varepsilon>0$ and $c\in\R^d$, including exceptional
parameters. The support-function convergence proved above and
Lemma~\ref{lem:linear-branches} finish the proof.
\end{proof}

\subsection{Applying the count to PSD factors}
We now put the slack factorization into the preceding setting.
Trace normalization supplies the sphere; the affine slack equations
supply the quartic penalty.

\begin{proof}[Proof of Theorem~\ref{thm:count}]
Put $M=rk-\binom{k}{2}$. If $N_{\le k}=0$, there is nothing to prove. Work in affine coordinates in which $P$ is full-dimensional. Boundedness implies that the facet normals positively span $\mathbb R^d$, so a strictly positive linear combination of the facet slacks is constant. That constant is positive by evaluation at an interior point. Rescaling gives positive weights $\lambda_i$ such that
\[
 \sum_i\lambda_i\ell_i(x)=1\quad\text{identically on }\mathbb R^d.
\]
Put $H=\sum_i\lambda_iA_i$. Compress to the orthogonal complement of its kernel if necessary. Each $A_i$ vanishes on that kernel, so all slack entries are preserved under this compression, and vertex-factor ranks cannot increase. A congruence transformation then makes $H=I$. The effective matrix order may have decreased to $r'\le r$ and the effective rank cap to $k'=\min\{k,r'\}$; the product $\prod_{i=r'-k'+1}^{r'}\frac{3^i+1}{2}$ is at most
$\prod_{i=r-k+1}^{r}\frac{3^i+1}{2}$, since its factors are increasing with
their indices and exceed one. It therefore suffices to prove the bound after compression. Reset $r,k$ to the effective values and $M=rk-\binom{k}{2}$. We now have
\begin{equation}\label{eq:normal}
 \sum_i\lambda_i A_i=I,\qquad \tr B_v=1.
\end{equation}

The slack map $x\mapsto(\ell_i(x))_i$ is an affine injection. Let $L$ be its affine image and choose affine equations for $L$, together with an affine map recovering $x$ from a slack vector in $L$.

By Lemma~\ref{lem:gram}, every rank-at-most-$k$ PSD matrix can be written as $YY^{\mathsf T}$ with $Y\in\mathbb R^{r\times k}$ and $Y_{ij}=0$ for $i<j$. The $M$ unrestricted entries form a vector $y\in\mathbb R^M$, partitioned
into column blocks of dimensions $m_j=r-j+1$. Substituting
\[
 z_i(y)=\tr(A_iYY^{\mathsf T})
\]
into the affine equations for $L$ gives polynomials $g_\alpha(y)$ of degree at most two. Substituting into the affine recovery map gives a quadratic map $\phi(y)$. Both are invariant under independent column signs, since these leave $YY^{\mathsf T}$ unchanged. Define
\[
 q(y)=\sum_\alpha g_\alpha(y)^2,\qquad
 T=\{y:\|y\|^2=1,\ q(y)=0\}.
\]
This is a nonnegative quartic of the type required by Lemma~\ref{lem:corners}.

If $y\in T$, then $(z_i(y))_i\in L$, and each $z_i(y)\ge0$. Thus $\phi(y)\in P$. Conversely, any chosen vertex factor of rank at most $k$ has a triangular Gram factor satisfying all $g_\alpha=0$ and $\|y\|^2=\tr B_v=1$, so that vertex belongs to $\phi(T)$. Lemma~\ref{lem:corners} proves the first inequality in \eqref{eq:low-rank-count}. For the second, repeatedly use
$(a+1)(b+1)\le2(ab+1)$ for $a,b\ge1$ to obtain
\[
 \prod_{j=1}^k\frac{3^{m_j}+1}{2}-1
 \le\frac{3^{\sum_jm_j}-1}{2}.\qedhere
\]
\end{proof}

\section{Applications}\label{sec:applications}
\subsection{From vertex counts to lift sizes}\label{sec:general-consequences}
The two useful specializations of Theorem~\ref{thm:count} are $k=1$
for SOS dimension and $k=r-d$ for unrestricted lifts. We record their
exact forms before applying them to the two families.

\begin{corollary}\label{cor:universal}
Let $P$ have $N$ vertices, dimension $d\ge1$, and $r=\xcPSD(P)$. Then
\begin{align}
 \sosc(P)&\ge\lceil\log_3(2N+1)\rceil,\label{eq:universal-sos}\\
 N&\le\prod_{i=d+1}^{r}\frac{3^i+1}{2}-1,\label{eq:universal-count}\\
 r&\ge\left\lceil
 \frac{\sqrt{(2d+1)^2+8\log_3(2N+1)}-1}{2}
 \right\rceil.\label{eq:universal-rank}
\end{align}
Writing $r=d+t$, bound \eqref{eq:universal-count} is equivalently
\begin{equation}\label{eq:excess-count}
 \log_3(N+1)\le dt+\frac{t(t+1)}2-t\log_3 2
 +\sum_{i=d+1}^{d+t}\log_3(1+3^{-i}).
\end{equation}
The same bound \eqref{eq:universal-count} holds for the number of facets.
\end{corollary}
\begin{proof}
Use $k=1$ in Theorem~\ref{thm:count} for an SOS factorization, and
$k=r-d$ for an unrestricted one. In the latter case,
\[
 rk-\binom{k}{2}=\binom{r+1}{2}-\binom{d+1}{2}
 =dt+\frac{t(t+1)}2.
\]
The second inequality in \eqref{eq:low-rank-count} gives
$2N+1\le3^{\binom{r+1}{2}-\binom{d+1}{2}}$; solving this quadratic
inequality gives \eqref{eq:universal-rank}. Polarity preserves dimension
and PSD rank and interchanges facets and vertices.
\end{proof}

\paragraph{Comparison with the displayed vertex bound.}
Fawzi--Safey El Din \cite[Theorem 4]{FawziSafey2018} prove that a closed
convex body containing the origin in its interior and admitting a
size-$r$ semidefinite lift has at most $2^{r^2}$ vertices, where a vertex
means a point with a full-dimensional normal cone. For polytopes this
also bounds facets by polarity, working in the affine hull and
translating an interior point to the origin. In base two, \eqref{eq:universal-count} reads
\begin{equation}\label{eq:comparison-fse}
 \log_2(N+1)\le\frac{\log_2 3}{2}
 \bigl(r^2+r-d^2-d\bigr)-(r-d)
 +\sum_{i=d+1}^{r}\log_2(1+3^{-i}).
\end{equation}
Thus the leading quadratic coefficient is
$\tfrac12\log_2 3=0.79248\ldots$, compared with $1$ in the
displayed Fawzi--Safey El Din bound. In particular, \eqref{eq:universal-count}
is strictly smaller than $2^{r^2}$ for every $1\le d<r$, since
\[
 \prod_{i=d+1}^r\frac{3^i+1}{2}
 <\prod_{i=d+1}^r2^{2i-1}=2^{r^2-d^2}<2^{r^2}.
\]
The dimension saving matters especially when $r$ is close to $d$:
\eqref{eq:excess-count} controls the excess $r-d$ directly.
Our theorem concerns polytopes; it does not replace the more general
convex-body and algebraic-boundary results of \cite{FawziSafey2018}.

The comparison with $2^{r^2}$ does not exhaust the consequences of that
work. Fawzi--Safey El Din explicitly observe that the B\'ezout estimate
in their proof can be replaced by the sum of the SDP algebraic degrees
over the admissible ranks \cite[Section 5]{FawziSafey2018}.
Write $\delta(m,4,q)$ for the algebraic degree with $m$ affine constraints
and matrix rank $q$. The nonzero entries in Nie--Ranestad--Sturmfels
\cite[Table 2]{NRS2010} give
\[
\begin{array}{c|rrrrrrrrr}
m&1&2&3&4&5&6&7&8&9\\ \hline
q=1&0&0&0&0&0&8&16&12&4\\
q=2&0&0&10&30&42&30&10&0&0\\
q=3&4&12&16&8&0&0&0&0&0\\ \hline
\sum_q\delta(m,4,q)&4&12&26&38&42&38&26&12&4
\end{array}
\]
The maximum is $42$, giving at most $42$ vertices for polygons of PSD
rank four by their degree-sum refinement.

The corner analysis in Theorem~\ref{cor:rank-four} gives $12$, including
when the rank-one loci contain curves.
We make no claim that \eqref{eq:universal-count} dominates these degree-based
bounds. The main assertion here is \eqref{eq:low-rank-count}, which
counts factors of rank at most $k$ within an arbitrary fixed
factorization, including one with higher-rank factors at other vertices.

\FloatBarrier
\subsection{Permutahedra}
The standard permutahedron is
\[
 \Pi_n=\conv\{(\sigma(1),\ldots,\sigma(n)):\sigma\in\mathfrak S_n\}.
\]
It has $n!$ vertices and dimension $n-1$. Its affine hull is
$\sum_i x_i=\binom{n+1}{2}$, and its facet slacks are
\begin{equation}\label{eq:slack}
 \ell_S(x)=\sum_{i\in S}x_i-\binom{|S|+1}{2},
 \qquad\varnothing\ne S\subsetneq[n].
\end{equation}
Goemans proves $\xcLP(\Pi_n)=\Theta(n\log n)$ \cite{Goemans2015}.
A hexagonal face gives the unrestricted lower bound
$\xcPSD(\Pi_n)\ge n+1$ for $n\ge3$: fixing coordinates
$4,\ldots,n$ to $4,\ldots,n$ leaves a regular hexagon of PSD rank four
\cite{GRT2013}, and Lemma~\ref{lem:faces} adds its codimension $n-3$.
Figure~\ref{fig:perm-face} follows the successive facet restrictions down to this hexagon.

\begin{figure}[tbp]
\centering
\begin{tikzpicture}[font=\small,line join=round,>=Stealth]
\node[text=figblue] at (0,2.35) {$F_n=\Pi_n$};
\node[text=figblue] at (5.5,2.35) {$F_4\cong\Pi_4$};
\node[text=figteal] at (10.7,2.35) {$F_3\cong\Pi_3$};
\node[align=center,text=figblue] at (0,.35) {$F_n\supset F_{n-1}$\\[5pt]$\supset\ \cdots$};
\node[font=\footnotesize,text=black!60] at (0,-.65) {face flag};
\draw[->,figblue!75,line width=.7pt] (1.35,.35)--(3.55,.35);
\node[align=center,font=\footnotesize] at (2.45,.93) {$x_n=n,\ldots,$\\$x_5=5$};
\node[align=center,font=\footnotesize,text=black!65] at (2.45,-.15) {$n-4$ facet steps};
\begin{scope}[shift={(5.5,0)},scale=.79]
\coordinate (p0) at (0.047992,0.610135);
\coordinate (p1) at (0.480002,-0.733229);
\coordinate (p2) at (-1.200043,0.991131);
\coordinate (p3) at (-0.336024,-1.695595);
\coordinate (p4) at (-2.016069,0.028764);
\coordinate (p5) at (-1.584060,-1.314599);
\coordinate (p6) at (1.248036,1.038746);
\coordinate (p7) at (1.680045,-0.304617);
\coordinate (p8) at (-1.248036,1.800739);
\coordinate (p9) at (0.047992,-2.229350);
\coordinate (p10) at (-2.064062,0.838372);
\coordinate (p11) at (-1.200043,-1.848354);
\coordinate (p12) at (1.200043,1.848354);
\coordinate (p13) at (2.064062,-0.838372);
\coordinate (p14) at (-0.047992,2.229350);
\coordinate (p15) at (1.248036,-1.800739);
\coordinate (p16) at (-1.680045,0.304617);
\coordinate (p17) at (-1.248036,-1.038746);
\coordinate (p18) at (1.584060,1.314599);
\coordinate (p19) at (2.016069,-0.028764);
\coordinate (p20) at (0.336024,1.695595);
\coordinate (p21) at (1.200043,-0.991131);
\coordinate (p22) at (-0.480002,0.733229);
\coordinate (p23) at (-0.047992,-0.610135);
\draw[black!23,densely dashed,line width=.4pt] (p0)--(p1);
\draw[black!23,densely dashed,line width=.4pt] (p0)--(p2);
\draw[black!23,densely dashed,line width=.4pt] (p0)--(p6);
\draw[black!23,densely dashed,line width=.4pt] (p1)--(p3);
\draw[black!23,densely dashed,line width=.4pt] (p1)--(p7);
\draw[black!23,densely dashed,line width=.4pt] (p2)--(p4);
\draw[black!23,densely dashed,line width=.4pt] (p2)--(p8);
\draw[black!23,densely dashed,line width=.4pt] (p3)--(p5);
\draw[black!23,densely dashed,line width=.4pt] (p3)--(p9);
\draw[black!23,densely dashed,line width=.4pt] (p4)--(p5);
\draw[black!23,densely dashed,line width=.4pt] (p4)--(p10);
\draw[black!23,densely dashed,line width=.4pt] (p5)--(p11);
\draw[black!23,densely dashed,line width=.4pt] (p6)--(p7);
\draw[black!23,densely dashed,line width=.4pt] (p6)--(p12);
\draw[black!23,densely dashed,line width=.4pt] (p7)--(p13);
\draw[black!23,densely dashed,line width=.4pt] (p8)--(p10);
\draw[black!23,densely dashed,line width=.4pt] (p8)--(p14);
\draw[black!23,densely dashed,line width=.4pt] (p9)--(p11);
\draw[black!23,densely dashed,line width=.4pt] (p9)--(p15);
\draw[black!23,densely dashed,line width=.4pt] (p10)--(p16);
\draw[black!23,densely dashed,line width=.4pt] (p11)--(p17);
\draw[black!23,densely dashed,line width=.4pt] (p12)--(p14);
\draw[black!23,densely dashed,line width=.4pt] (p12)--(p18);
\draw[black!23,densely dashed,line width=.4pt] (p13)--(p15);
\draw[black!23,densely dashed,line width=.4pt] (p13)--(p19);
\draw[black!23,densely dashed,line width=.4pt] (p14)--(p20);
\draw[black!23,densely dashed,line width=.4pt] (p15)--(p21);
\draw[black!23,densely dashed,line width=.4pt] (p16)--(p17);
\draw[black!23,densely dashed,line width=.4pt] (p16)--(p22);
\draw[black!23,densely dashed,line width=.4pt] (p17)--(p23);
\draw[black!23,densely dashed,line width=.4pt] (p18)--(p19);
\draw[black!23,densely dashed,line width=.4pt] (p18)--(p20);
\draw[black!23,densely dashed,line width=.4pt] (p19)--(p21);
\draw[black!23,densely dashed,line width=.4pt] (p20)--(p22);
\draw[black!23,densely dashed,line width=.4pt] (p21)--(p23);
\draw[black!23,densely dashed,line width=.4pt] (p22)--(p23);
\filldraw[fill=figblue!7,draw=figblue!65,line width=.6pt] (p5)--(p3)--(p1)--(p0)--(p2)--(p4)--cycle;
\filldraw[fill=figblue!7,draw=figblue!65,line width=.6pt] (p7)--(p13)--(p19)--(p18)--(p12)--(p6)--cycle;
\filldraw[fill=figblue!15,draw=figblue!65,line width=.6pt] (p1)--(p7)--(p6)--(p0)--cycle;
\filldraw[fill=figblue!15,draw=figblue!65,line width=.6pt] (p10)--(p4)--(p2)--(p8)--cycle;
\filldraw[fill=figblue!15,draw=figblue!65,line width=.6pt] (p11)--(p9)--(p3)--(p5)--cycle;
\filldraw[fill=figteal!25,draw=figteal,line width=1.1pt] (p2)--(p0)--(p6)--(p12)--(p14)--(p8)--cycle;
\node[font=\scriptsize,text=figteal] at (0.00000,1.41974) {$F_3$};
\filldraw[fill=figblue!7,draw=figblue!65,line width=.6pt] (p3)--(p9)--(p15)--(p13)--(p7)--(p1)--cycle;
\end{scope}
\draw[->,figteal,line width=.8pt] (7.45,.35)--(9.15,.35);
\node[font=\footnotesize] at (8.3,.8) {$x_4=4$};
\node[font=\footnotesize,text=black!65] at (8.3,-.1) {one facet step};
\begin{scope}[shift={(10.7,.35)},scale=1.13]
\filldraw[fill=figteal!13,draw=figteal,line width=1.1pt] (0:1)--(60:1)--(120:1)--(180:1)--(240:1)--(300:1)--cycle;
\foreach \ang in {0,60,120,180,240,300} {\fill[figteal] (\ang:1) circle (1.3pt);}
\end{scope}
\node[font=\footnotesize,text=black!65] at (0,-2.13) {codimension $0$};
\node[font=\footnotesize,text=black!65] at (5.5,-2.13) {codimension $n-4$};
\node[font=\footnotesize,text=black!65] at (10.7,-2.13) {codimension $n-3$};
\draw[black!15] (-1.2,-2.52)--(12,-2.52);
\node[font=\footnotesize,text=black!65] at (5.4,-2.93) {Each facet restriction lowers the available matrix size by at least one.};
\node[text=figblue] at (0,-3.55) {$r$};
\node[text=figblue] at (5.5,-3.55) {$\le r-(n-4)$};
\node[text=figteal] at (10.7,-3.55) {$4\le r-(n-3)$};
\draw[->,black!45] (1.35,-3.55)--(3.55,-3.55);
\draw[->,black!45] (7.45,-3.55)--(9.15,-3.55);
\node at (5.4,-4.4) {$\displaystyle\xcPSD(\Pi_n)\ge \underbrace{4}_{\text{hexagon}}+\underbrace{(n-3)}_{\text{facet steps}}=n+1$};
\end{tikzpicture}
\caption{Codimension accumulates along the face flag $F_m=\{x\in\Pi_n:x_{m+1}=m+1,\ldots,x_n=n\}\cong\Pi_m$. Each step is a facet restriction. The higher-dimensional steps are schematic; $F_4$ is drawn with $F_3$ highlighted. A size-$r$ lift yields a lift of $F_3$ of size at most $r-(n-3)$, while the hexagon requires size four. For $n=4$ the first stage is absent; for $n=3$ only the hexagon remains.}
\label{fig:perm-face}
\end{figure}
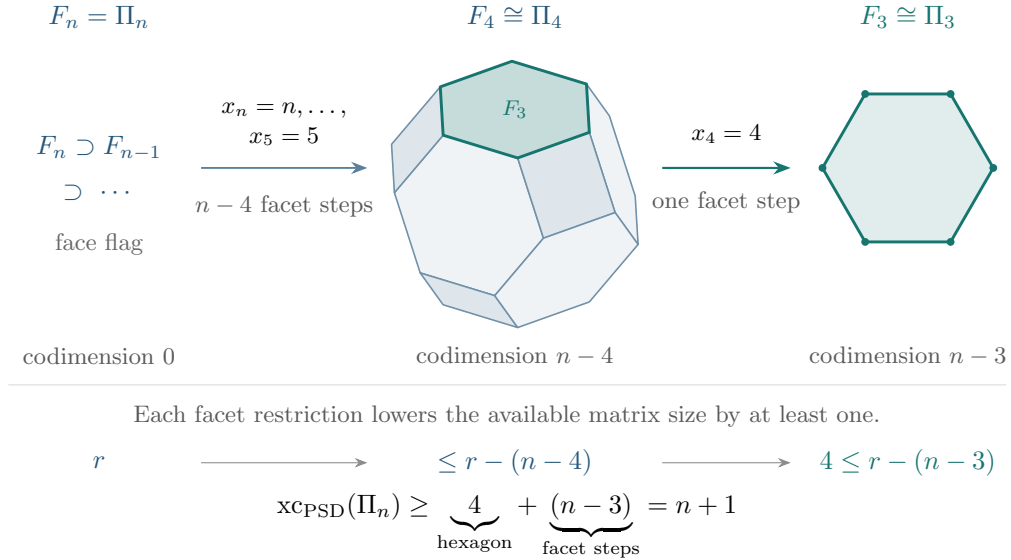

\begin{samepage}
\begin{theorem}\label{thm:perm}
For $n\ge3$,
\begin{equation}\label{eq:perm-explicit}
 \xcPSD(\Pi_n)\ge\mathcal L_n:=
 \min\left\{r\in\mathbb Z:r\ge n+1,\quad
 n!+1\le\prod_{i=n}^{r}\frac{3^i+1}{2}\right\}.
\end{equation}
In particular,
\[
 \sosc(\Pi_n)=\Theta(n\log n),\qquad
 \xcPSD(\Pi_n)\ge n+\log_3n-O(1).
\]
\end{theorem}
\end{samepage}
\begin{proof}
Corollary~\ref{cor:universal} gives
$\sosc(\Pi_n)\ge\lceil\log_3(2n!+1)\rceil$ and, with $d=n-1$,
the product condition in \eqref{eq:perm-explicit}. The hexagonal-face
bound gives $r\ge n+1$. Proposition~\ref{prop:sos-rank} and Goemans' LP upper
bound prove the matching SOS upper bound.

Put $b_n=\log_3(2n!+1)$ and write $r=n-1+t$. The second inequality
in \eqref{eq:low-rank-count} gives
\[
 (n-1)t+\frac{t(t+1)}2\ge b_n.
\]
Its positive root in $t$ differs from $b_n/(n-1)$ by
$O((\log n)^2/n)$. Stirling's formula gives
$b_n/(n-1)=\log_3 n-1/\ln3+o(1)$, proving the claimed additive bound.
\end{proof}

\begin{remark}\label{rem:thresholds}
The first five improvements over the face bound are
\[
\begin{array}{c|rrrrr}
\text{lower bound}&n+2&n+3&n+4&n+5&n+6\\ \hline
\text{first }n&22&72&221&664&1990
\end{array}
\]
These thresholds require only integer arithmetic. For every integer $j\ge2$,
\begin{equation}\label{eq:threshold-integer}
 \mathcal L_n\ge n+j
 \quad\Longleftrightarrow\quad
 2^j(n!+1)>\prod_{i=n}^{n+j-1}(3^i+1).
\end{equation}
Indeed, exclude the integer matrix order $r=n+j-1$ in
\eqref{eq:universal-count}. For fixed $j$, let $\eta_j(n)$ be the ratio
of the left side of the strict inequality to its right side. Then
\[
 \frac{\eta_j(n+1)}{\eta_j(n)}
 =\frac{(n+1)!+1}{n!+1}\frac{3^n+1}{3^{n+j}+1}
 >\frac{n}{3^j}\qquad(n\ge2).
\]
Exact comparisons for $3\le n\le\max\{3^j,n_j\}$, where $n_j$ is
the stated threshold, certify the first occurrences and persistence
through this range. The ratio is increasing thereafter. Each extra unit of excess asymptotically requires
multiplying $n$ by three.
\end{remark}

The unrestricted upper bound remains $O(n\log n)$. Theorem~\ref{thm:perm}
does not determine its asymptotic order, whereas it does determine the
order of SOS spaces without any symmetry or degree restriction.

\subsection{Regular polygons}
For the regular $N$-gon $R_N$, the linear formulations of
Fiorini--Rothvo\ss--Tiwary give $\xcLP(R_N)=O(\log N)$ \cite{FRT2012}.
Together with \eqref{eq:universal-sos} and
Proposition~\ref{prop:sos-rank}, this proves
\begin{equation}\label{eq:polygon-sos}
 \sosc(R_N)=\Theta(\log N).
\end{equation}
For every convex $N$-gon $P$, regular or otherwise,
\begin{equation}\label{eq:polygon-rank}
 \xcPSD(P)\ge
 \left\lceil\frac{\sqrt{25+8\log_3(2N+1)}-1}{2}\right\rceil.
\end{equation}
Thus the unrestricted lower bound still has order $\sqrt{\log N}$,
while its leading constant and finite values improve on the estimate
$\sqrt{\log_2 N}$ displayed in \cite{FawziSafey2018}. Their
algebraic-degree refinement can give stronger finite bounds, as noted
in Section~\ref{sec:general-consequences}.
For generic polygons, Gouveia--Robinson--Thomas
\cite{GouveiaRobinsonThomas2015} give the stronger lower bound
$\Omega(N^{1/4})$; genericity is essential to that statement.
Goucha--Gouveia--Silva \cite{GouchaGouveiaSilva2017} study nonnegative,
Boolean, and complex PSD ranks of regular polygons. The PSD rank here
is over the reals. For linear extension complexity in low dimensions,
see also Kwan--Sauermann--Zhao \cite{KwanSauermannZhao2022}.

\subsection{Bounded vertex ranks and rank distributions}\label{sec:rank-distribution}
The restriction defining $r_k(P)$ separates the cost of the lift from
the ranks used to represent its vertices. It has a polar version:
after placing the origin in the interior, transposing
the slack factorization identifies the minimum size with facet ranks
at most $k$ for $P$ with $r_k(P^\circ)$.

\begin{corollary}\label{cor:bounded-vertex-ranks}
For every fixed $k\ge1$,
\[
 r_k(\Pi_n)=\Theta(n\log n),\qquad r_k(R_N)=\Theta(\log N).
\]
The constants may depend on $k$.
\end{corollary}
\begin{proof}
For a polytope $P$ with $N$ vertices, put $r=r_k(P)$ and
$\kappa=\min\{k,r\}$. The second inequality in
Theorem~\ref{thm:count} gives
\[
 r_k(P)\ge\frac{\log_3(2N+1)+\binom{\kappa}{2}}{\kappa}.
\]
In particular, when $k\le r_k(P)$,
\begin{equation}\label{eq:bounded-rank-lower}
 r_k(P)\ge\frac{\log_3(2N+1)+\binom{k}{2}}{k}.
\end{equation}
For fixed $k$, this condition holds for all sufficiently large $N$,
since Corollary~\ref{cor:universal} forces $r_k(P)\to\infty$ as
$N\to\infty$. Apply \eqref{eq:bounded-rank-lower} with $N=n!$ for
the permutahedron and with $N$ for the regular polygon. The LP
upper bounds yield rank-one vertex factors by
Proposition~\ref{prop:sos-rank}, and hence apply for every $k\ge1$.
\end{proof}

\begin{corollary}\label{cor:rank-distribution}
For every fixed factorization, the following hold.
\begin{enumerate}[label=\textup{(\roman*)},leftmargin=2em]
\item Fix $C>0$. In a size-$r$ factorization of $\Pi_n$ with $r\le Cn$,
put $k_n=\lfloor\log_3 n/(2C)\rfloor$. For sufficiently large $n$,
\begin{equation}\label{eq:perm-rank-tail}
 \frac{\bigl|\{\sigma:\rank B_\sigma\le k_n\}\bigr|}{n!}
 \le3^{-\frac12n\log_3 n+O(n)}.
\end{equation}
Thus all but a vanishing fraction of the vertex factors have rank
$\Omega(\log n)$.
\item In a size-$r$ factorization of any $N$-vertex polytope, if
$1\le k\le r$ and $rk\le\tfrac12\log_3N$, then
\begin{equation}\label{eq:polygon-rank-tail}
 \bigl|\{v:\rank B_v\le k\}\bigr|\le\frac{\sqrt N-1}{2}.
\end{equation}
Consequently, if $r=o(\log N)$ along a sequence of such factorizations,
the fraction of vertex factors of rank at most any fixed $K$ tends to zero.
\end{enumerate}
\end{corollary}
\begin{proof}
For (i), the dimension bound ensures $1\le k_n\le r$ for large $n$.
If $N_{\le k_n}>0$, Theorem~\ref{thm:count} and Stirling's formula give
\[
 \log_3\frac{N_{\le k_n}}{n!}
 \le rk_n-\log_3(n!)
 \le-\tfrac12n\log_3n+O(n).
\]
If $N_{\le k_n}=0$, the conclusion is immediate. For (ii),
$N_{\le k}\le(3^{rk}-1)/2\le(\sqrt N-1)/2$.
For the last assertion, take $k=\min\{K,r\}$, so $rk\le rK=o(\log N)$
and $N_{\le K}=N_{\le k}$. Divide by $N$.
\end{proof}
Thus an asymptotic improvement over the SOS order for either family
requires unbounded ranks at almost every vertex. These estimates do
not rule out such factorizations.

\section{Rank-one factors and polygon corners}\label{sec:rank-four-corners}
The vertex count controls how many low-rank factors a lift can have.
For polygons, the geometry also restricts their arrangement. A curve
whose image stays inside a polygon must stop to first order when it
reaches a corner. Two distinct endpoints, in turn, must remain distinct
after compression to the relevant kernel.

These observations lead to two structural conclusions. The first
possible linear dependence among rank-one corner matrices has eight
terms and comes from two orthonormal bases. A signed representation of
a facet factor needs enough positive and negative terms to preserve
both endpoints. The octagon construction in Section~\ref{sec:octagons} realizes
the eight-term dependence explicitly.

\subsection{The two factor spaces}
\begin{proposition}\label{prop:p3-spans}
Suppose a polygon $P$ with $N\ge5$ vertices has a size-four PSD slack
factorization $S_{ij}=\tr(A_iB_j)$. Set
$U=\operatorname{span}\{A_i\}$ and $W=\operatorname{span}\{B_j\}$ in
$\mathbb S^4$. Put $u=\dim U$ and $w=\dim W$. Then
\begin{equation}\label{eq:factor-spans}
 u+w\le13.
\end{equation}
Consequently, either $P$ or its polar has a compact size-four
spectrahedral lift of affine dimension at most five.
\end{proposition}
\begin{proof}
The trace pairing restricted to $U\times W$ has rank
$\rank S=3$, since the factors span these spaces. Hence
$\dim(W\cap U^\perp)=\dim W-3$. As $\dim\mathbb S^4=10$,
\[
 \dim W-3\le\dim U^\perp=10-\dim U,
\]
which proves the inequality. Both factor families have positive definite
sums: otherwise compression would give a size-three factorization,
impossible for $N\ge5$ by \cite[Theorem 4.7]{GRT2013}.
Write the facet slacks as $\ell_i(x)$ and choose positive weights
$\lambda_i$ such that $\sum_i\lambda_i\ell_i(x)=1$ identically.
Such weights exist because $P$ is bounded and full-dimensional in its
affine hull. Put $H=\sum_i\lambda_i A_i\succ0$; positivity follows
from the positive definite sum and the strictly positive weights.

Consider
\[
 Q=\{B\in W:B\succeq0,\ \tr(HB)=1\}.
\]
It is compact because $H\succ0$. The average of the normalized vertex
factors is positive definite, so its affine dimension is $w-1$.
If $B=\sum_j t_jB_j$ lies in $Q$, then
$\sum_jt_j=1$, and
\[
 \tr(A_iB)=\ell_i\Bigl(\sum_jt_jv_j\Bigr).
\]
The resulting point is independent of the chosen coefficients, since
the facet normals span the dual of the affine hull. This defines an
affine projection of $Q$ into $P$, by nonnegativity of every trace.
Every $v_j$ lifts to $B_j$, so convexity gives equality of the image
with $P$. Applying the same construction to the polar gives a normalized
PSD slice in $U$ mapping onto $P^\circ$. Since $\min(u,w)\le6$,
one of these lifts has affine dimension at most five.
\end{proof}

\subsection{Corners on curves}
Let $L\subseteq\mathbb S^r$ be a real linear space, let $H\succ0$, and set
\[
 K_L=\{B\in L\cap\PSD r:\tr(HB)=1\}.
\]
Suppose an affine map $\pi$ sends $K_L$ into a full-dimensional polygon $P$.
Write $c(L)$ for the number of vertices of $P$ attained by rank-one matrices
in $K_L$. The associated complex projective locus is
\[
 X_L=\{[z]\in\mathbb P^{r-1}_{\mathbb C}:zz^{\mathsf T}\in L_{\mathbb C}\}.
\]
On its real points the map is
\begin{equation}\label{eq:corner-map}
 \phi([z])=\pi\left(\frac{zz^{\mathsf T}}{z^{\mathsf T}Hz}\right).
\end{equation}
Its homogeneous coordinates are three real quadrics, with the denominator
$z^{\mathsf T}Hz$ strictly positive at every real projective point.
The image of every real point of $X_L$ is in $P$.
For a linear subspace $E\subseteq\R^r$, write $\mathbb S(E)$ for
the space of real symmetric matrices whose image is contained in $E$.

\begin{lemma}\label{lem:corner-linear}
If $\mathbb S(E)\subseteq L$, where $E\subseteq\R^r$ has dimension $h$,
at most $h$ vertices of $P$ are attained by rank-one matrices supported on $E$.
\end{lemma}
\begin{proof}
Normalize $H|_E$ to the identity by congruence. The two coordinates of
\eqref{eq:corner-map} on $\mathbb P(E)$ are Rayleigh quotients of symmetric matrices
$D_1,D_2$. If a unit vector $z$ maps to a vertex $v$, it maximizes
$z^{\mathsf T}(a_1D_1+a_2D_2)z$ for every $a$ in the interior of the normal
cone of $v$. Hence $z$ is an eigenvector for two linearly independent such
combinations, and thus a common eigenvector of $D_1,D_2$.
Different vertices give different pairs of eigenvalues. Their chosen common
eigenvectors are orthogonal, because at least one of the two symmetric
matrices has different eigenvalues on them. There are at most $h$.
\end{proof}

\begin{lemma}\label{lem:corner-curve}
Let $C\subseteq X_L$ be a geometrically integral real projective curve of
degree $\delta$, and let its normalization have genus $g$.
If the complex image of \eqref{eq:corner-map} on $C$ is not a line or a point, the
number of vertices of $P$ attained at real points of the normalization is at
most
\[
 3\delta+3g-3.
\]
If the image is a line or a point, the number is at most two.
\end{lemma}
\begin{proof}
Pull the three quadratic sections back to the normalization and remove
their common base divisor. They give a basepoint-free linear series of
projective dimension two and degree $e\le2\delta$. There are no real base
points, since the denominator is positive. At a real normalization point
mapping to a vertex $v$, a local real parameter takes both signs. The two
opposite first derivatives would both have to belong to the pointed tangent
cone of $P$ at $v$. Therefore the derivative is zero; Figure~\ref{fig:corner-stationarity}
shows this local obstruction.

The vanishing sequence of the three sections is consequently at least
$(0,2,3)$, with ramification weight at least two. The Wronskian is a nonzero
section of $\mathcal L^{\otimes3}\otimes K^{\otimes3}$, so its zero divisor
has degree $3e+6g-6$. Equivalently this is the Pl\"ucker formula for a
$g^2_e$; see \cite[Section 4.2]{Esteves2006}. Counting one real preimage of each
distinct vertex gives the bound. If the complex image is a line, that line
contains at most two vertices of a polygon. A constant image contributes
at most one.
\end{proof}

In particular, an entire real projective line, a smooth conic, and a twisted
cubic contained in $X_L$ contribute at most $2,3,6$ vertices, respectively.
For the line, use Lemma~\ref{lem:corner-linear}; the other two curves have genus zero.
These bounds count corners, not all points of the curves.

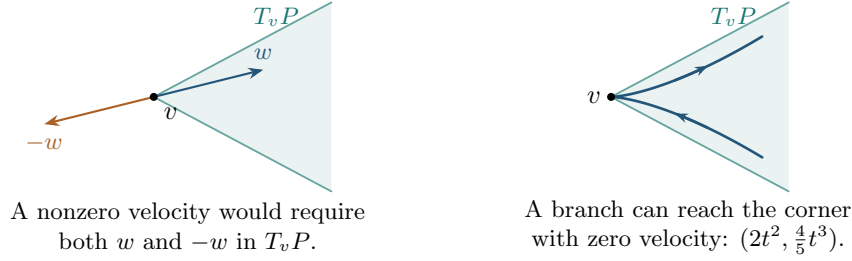
\begin{figure}[tbp]
\centering
\begin{tikzpicture}[font=\small,line cap=round,line join=round]
\begin{scope}
\fill[figteal!9] (0,0)--(2.35,1.25)--(2.35,-1.25)--cycle;
\draw[figteal!65,line width=.7pt] (2.35,1.25)--(0,0)--(2.35,-1.25);
\node[text=figteal,font=\footnotesize] at (1.62,1.06) {$T_vP$};
\draw[-{Stealth[length=2mm]},figblue,thick] (0,0)--(1.45,.36)
 node[above,font=\footnotesize] {$w$};
\draw[-{Stealth[length=2mm]},figorange,thick] (0,0)--(-1.45,-.36)
 node[below,font=\footnotesize] {$-w$};
\fill (0,0) circle (1.5pt);
\node[below right] at (0,0) {$v$};
\node[align=center,font=\footnotesize] at (.45,-1.72)
 {A nonzero velocity would require\\both $w$ and $-w$ in $T_vP$.};
\end{scope}
\begin{scope}[shift={(6.05,0)}]
\fill[figteal!9] (0,0)--(2.35,1.25)--(2.35,-1.25)--cycle;
\draw[figteal!65,line width=.7pt] (2.35,1.25)--(0,0)--(2.35,-1.25);
\node[text=figteal,font=\footnotesize] at (1.62,1.06) {$T_vP$};
\draw[figblue,line width=1pt,domain=-1:1,samples=81,variable=\t]
 plot ({2*\t*\t},{.8*\t*\t*\t});
\draw[-{Stealth[length=1.7mm]},figblue,line width=1pt,domain=-.8:-.65,samples=12,variable=\t]
 plot ({2*\t*\t},{.8*\t*\t*\t});
\draw[-{Stealth[length=1.7mm]},figblue,line width=1pt,domain=.65:.8,samples=12,variable=\t]
 plot ({2*\t*\t},{.8*\t*\t*\t});
\fill (0,0) circle (1.5pt);
\node[left] at (0,0) {$v$};
\node[align=center,font=\footnotesize] at (1.02,-1.72)
 {A branch can reach the corner\\with zero velocity: $(2t^2,\frac45t^3)$.};
\end{scope}
\end{tikzpicture}
\caption{Why corners force ramification. A local real parameter on the
normalization takes both signs. If $\gamma$ parametrizes the image near
a corner $v=\gamma(0)$, containment in $P$ gives
$\pm\gamma'(0)\in T_vP$. The tangent cone is pointed,
hence $\gamma'(0)=0$. The cusp on the right illustrates one way a branch
can remain in the cone on both sides of the parameter.}
\label{fig:corner-stationarity}
\end{figure}
\subsection{The first possible dependence}
A linear dependence among rank-one PSD matrices equates two positive
sums. The two sums have the same range. At a polygon corner, this simple
observation is surprisingly restrictive.

\begin{theorem}\label{lem:corner-independent}
Any collection of at most seven rank-one matrices in $K_L$ mapping to
distinct vertices of $P$ is linearly independent. Consequently,
$c(L)\le\dim L$ if $\dim L\le6$.

If eight such matrices are dependent, they span a seven-dimensional
space. Their vectors have a common four-dimensional span, and after a
congruence on that span their rays are the projectors onto two
orthonormal bases.
\end{theorem}
\begin{proof}
Choose a minimally dependent subcollection $z_i z_i^{\mathsf T}$ of size
$m\le7$, and split its dependence into two positive sums. Their common
range $E$ has dimension $h$, where
\[
 2h\le m,\qquad m-1\le\dim\mathbb S(E)=\binom{h+1}{2}.
\]
Thus $h\le3$. If $h=1$, all the normalized matrices have the same image.
If $h=2$, the inequalities force $m=4$, and the matrices span
$\mathbb S(E)$. Lemma~\ref{lem:corner-linear} allows only two corner
images on this space.

If $h=3$, then $m=6$ or $7$. For $m=7$ the matrices again span
$\mathbb S(E)$, which allows only three corner images. For $m=6$ their
span is a hyperplane in $\mathbb S(E)$, so its rank-one locus in
$\mathbb P(E)$ is a conic. A smooth conic contributes at most three
corners by Lemma~\ref{lem:corner-curve}; two real lines contribute at
most four by Lemma~\ref{lem:corner-linear}. A double line contributes
at most two, and two conjugate nonreal lines have at most one real point.
None can supply six distinct corners. This proves independence.

Now suppose eight matrices are dependent. Any seven are independent,
so their dependence is unique up to scale and has full support. Splitting
signs gives a positive matrix $M$ on a common range $E$. Since their span
has dimension seven, $\dim E\ge4$. Each side of the dependence has at
least $\dim E$ summands, so there are exactly four on each side and
$\dim E=4$. Congruence by $M^{-1/2}$ on $E$, with the positive
coefficients absorbed into the vectors, turns each sum into a
four-vector decomposition of the identity. Each is an orthonormal basis.
\end{proof}

\subsection{Kernel directions at a facet}
The following argument uses both sides of the slack factorization.
It applies in every matrix size.

\begin{lemma}\label{lem:five-factors}
In a real PSD slack factorization of a polygon, of any size, suppose a
facet factor $A$ of rank $a$ is a linear combination of other rank-one
facet factors. Every such expression has at least $a+2$ positive
coefficients and at least two negative coefficients.
\end{lemma}
\begin{proof}
Write $A=\sum_i t_i a_i a_i^{\mathsf T}$, discarding zero coefficients,
and put $E=\operatorname{span}\{a_i\}$, $h=\dim E$.
There must be at least three terms. Indeed, two distinct facet vectors
cannot be proportional, and a PSD combination of at most two independent
rank-one matrices has only positive coefficients. Such a combination
cannot represent another facet: on the relative interior of that facet,
all the selected slacks are strictly positive.

Let $B,B'$ be the factors of the two endpoints of the facet of $A$.
Their compressions to $E$ are PSD and supported on $\ker(A|_E)$.
Three selected facet slacks are linearly independent: after placing the
origin inside the polygon and normalizing them to $1-c_i\cdot x$, the
$c_i$ are three distinct vertices of the polar polygon and are not
collinear. Neither compressed endpoint factor is zero, since three
distinct facets cannot meet at a vertex. If $\dim\ker(A|_E)\le1$, the
two compressions are positive multiples of each other. The three slacks,
and hence the homogeneous endpoint coordinates, would then be
proportional. The endpoints would coincide. Thus
\[
 h-a\ge2.
\]
Split the expression as $A=M_+-M_-$ according to the signs of $t_i$.
Since $M_+=A+M_-$, its range contains that of $M_-$ and therefore equals
$E$. The number of positive coefficients is at least $h\ge a+2$.
Also $h=\rank M_+\le a+\rank M_-$, so the number of negative
coefficients is at least $h-a\ge2$.
\end{proof}

The argument counts directions needed to distinguish the endpoints,
not just directions occupied by the facet factor. This is why a rank-$a$
factor needs at least $a+4$ terms in such a representation.
The next section combines these principles with the small-dimensional
rank-one loci to prove Theorem~\ref{cor:rank-four}.

\section{The twelve-vertex bound}\label{sec:rank-four-bound}
We now prove Theorem~\ref{cor:rank-four}. We count the polygon corners
attained by the rank-one loci: even when a locus contains a curve,
Lemma~\ref{lem:corner-curve} bounds its corner images. Thus
positive-dimensional loci become cases of the argument, and no
finiteness assumption is needed. We first bound the corner counts on
the small loci, then combine the two factor spaces using polygon
incidences. The only remaining equality case becomes a question about
cubic polynomials and is excluded in Section~\ref{sec:exclude-thirteen}.

\subsection{Counting corners on the small rank-one loci}
\begin{lemma}\label{lem:corner-loci}
For $L\subseteq\mathbb S^4$ one has
\[
 c(L)\le8\quad(\dim L\le7),\qquad
 c(L)\le12\quad(\dim L\le8).
\]
No finiteness assumption on $X_L$ is needed.
\end{lemma}
\begin{proof}
Put $D=L^\perp$, the space of quadrics defining $X_L$, so $\dim D\ge2$.
The argument has one degree budget: apart from a persistent integral
quartic, each component contributes at most twice its degree, and the
total degree is four.

First suppose all quadrics in $D$ share a factor. Their common factor
is a real linear form $\ell$: a quadratic common factor would force
$\dim D=1$, and a nonreal linear factor would bring its conjugate.
Write $D=\ell E$, where $E$ contains at least two independent linear
forms. Then $X_L$ is a real plane together with a possibly empty real
linear space of dimension at most one. Lemma~\ref{lem:corner-linear}
gives at most $3+2=5$ corner images.

Otherwise choose real $q_1,q_2\in D$ with no common factor. Their
intersection $Y\subset\mathbb P^3_{\mathbb C}$ is a pure one-dimensional
complete intersection of degree four. Its reduced irreducible components
have total degree at most four. Call a component persistent if it is
contained in $X_L$.

For the nonpersistent components, choose $q\in D$ that vanishes
identically on none of them. On each such component of degree $\delta$,
B\'ezout bounds the points of $X_L$ by $2\delta$.
A real geometrically integral persistent component of degree at most
three is a line, a smooth conic, or a twisted cubic. A planar cubic
cannot occur, since both $q_1,q_2$ would contain its plane and share a
factor. The preceding lemmas bound the corner contributions of these
three curves by $2,3,6$, respectively, each at most twice its degree.
Nonreal persistent components occur in conjugate pairs. Only pairs of
lines or conics fit within degree four; their real points lie in the
intersection of the pair. Two distinct lines have at most one common
point, and two distinct conics have at most four. These contributions
also fit the same degree budget. Summing gives $c(L)\le8$ unless a
persistent integral quartic occurs.

Such a quartic exhausts the degree of $Y$, with multiplicity one. A
complete intersection has no embedded components, so $Y$ is reduced
and equals that quartic. Its saturated ideal is $(q_1,q_2)$, with
degree-two part $\operatorname{span}(q_1,q_2)$. Since every member of
$D$ vanishes on the quartic, necessarily $\dim D=2$ and $X_L=Y$.
Thus this exception occurs only when $\dim L=8$.

The quartic has arithmetic genus one by the complete-intersection
Hilbert polynomial. If its normalization has genus one, it is smooth,
and Lemma~\ref{lem:corner-curve} gives twelve corner images. Otherwise
the normalization has genus zero and the total delta invariant is one.
There is at most one singular point, hence at most one real point
without a real preimage on the normalization. The same lemma, together
with that possible extra point, gives at most $9+1=10$. This proves
both bounds, including the singular and nonreduced cases.
\end{proof}

\subsection{Combining the two factor spaces}
Retain the factor spaces $U,W$ from Proposition~\ref{prop:p3-spans}, and
put $u=\dim U$, $w=\dim W$. Their normalized PSD slices map onto the
polar polygon and the polygon, respectively. Thus the corner bounds
apply to both sides. By Proposition~\ref{prop:rank-budget}, every factor
has rank one or two, and incident factors cannot both have rank two.

Let $f_k,\nu_k$ count rank-$k$ factors on the facet and vertex sides. If both
sides contain rank-two factors, let $\rho_F,\rho_V\ge1$ count their
respective cyclic runs. The $f_2$ facets in $\rho_F$ runs meet
$f_2+\rho_F$ vertices, all rank one. Interchanging the sides gives
\begin{equation}\label{eq:factor-runs}
 N\le f_1+\nu_1-\max(\rho_F,\rho_V)\le f_1+\nu_1-1.
\end{equation}
If one side is entirely rank one, $N$ is simply its rank-one count.

\begin{corollary}\label{cor:five-dim-facets}
In the size-four setting, if $N\ge6$ and $u=5$, then $f_1\le4$.
The analogous statement holds for $W$.
\end{corollary}
\begin{proof}
Theorem~\ref{lem:corner-independent} gives $f_1\le5$.
If equality holds, these five factors form a basis of $U$. Every other
facet factor has rank two. Lemma~\ref{lem:five-factors} requires at least
six terms in its expansion in that basis, a contradiction.
\end{proof}

For a set $T$ of facets, let $e(T)$ be the number of polygon vertices incident to at least one facet in $T$. Keeping track of these incidences sharpens the quartic bound.

\begin{lemma}\label{lem:elliptic-corners}\label{lem:quartic-refinement}
Suppose $X_W$ is a reduced geometrically integral complete intersection of two quadrics. If $T$ consists of three rank-one facets, then
\[
 \nu_1\le e(T)+4.
\]
If $X_W$ is smooth and $T$ consists of two rank-one facets, then $\nu_1\le e(T)+8$.
\end{lemma}
\begin{proof}
For three facets write their factors as $a_i a_i^{\mathsf T}$, $1\le i\le3$. The vectors $a_i$ are independent: otherwise their three distinct projectors would span the symmetric matrices on a two-dimensional space, contrary to Lemma~\ref{lem:corner-linear}. The three facet slacks are also independent, since three distinct vertices of the polar polygon are not collinear. In these coordinates the corner map is the linear projection
\[
 \rho:X_W\longrightarrow\mathbb P^2,\qquad [z]\longmapsto[a_1^{\mathsf T}z:a_2^{\mathsf T}z:a_3^{\mathsf T}z],
\]
followed by coordinatewise squaring. Its center is a real projective point outside $X_W$: otherwise a real rank-one matrix would map to a point on three distinct polygon facets.

On the normalization, $\rho$ defines a basepoint-free series of degree four. Its image is not a line, because $X_W$ is nondegenerate. Thus it is birational onto a plane quartic or a double cover of a conic. Write $g\in\{0,1\}$ for the normalization genus. There are at most $1-g$ real points of $X_W$ without real preimages on the normalization, since its arithmetic genus is one.

In the birational case the plane quartic has total delta invariant $3-g$. Points with zero differential map to singularities, so there are at most $3-g$ such images. Adding the at most $1-g$ exceptional real points gives at most $4-2g\le4$ possible corner images away from $T$. In the double-cover case, Riemann--Hurwitz gives $2g+2$ ramification points, and adding the exceptional real points gives at most $g+3\le4$. Coordinatewise squaring is unramified off the three coordinate lines, while the derivative of the corner map vanishes at every real normalization point mapping to a corner, by Lemma~\ref{lem:corner-curve}. This proves the first assertion.

For two facets, with vectors $a,b$, consider the pencil $[a^{\mathsf T}z:b^{\mathsf T}z]$ on the smooth genus-one curve. Removing its base divisor gives a nonconstant map to $\mathbb P^1$ of degree at most four, hence with at most eight ramification points. At a corner off these facets, the derivative of the ratio of their slacks is zero. That ratio is $(a^{\mathsf T}z/b^{\mathsf T}z)^2$, whose square root is finite and nonzero there, so the pencil is ramified. This bounds the corners off $T$ by eight.
\end{proof}

\begin{corollary}\label{cor:five-eight}
If $(u,w)=(5,8)$, then $N\le12$. If at least three facet factors have rank one, then $N\le11$.
\end{corollary}
\begin{proof}
Assume $N\ge6$ and put $f=f_1\le4$. If $0<f<N$, let $\rho_F$ count the cyclic runs of rank-one facets, equivalently of rank-two facets. A rank-two vertex can occur only between two rank-one facets, so
\begin{equation}\label{eq:adjacency}
 \nu_2\le f-\rho_F,\qquad N\le \nu_1+f-\rho_F.
\end{equation}
If $X_W$ is not a reduced geometrically integral quartic, the proof of Lemma~\ref{lem:corner-loci} gives $\nu_1\le8$, and hence $N\le11$ when $f>0$; if $f=0$, all vertices are rank one and $N\le8$. We may therefore assume that $X_W$ is an integral quartic.

For $f=3$, take all three facets as $T$, so $e(T)=3+\rho_F$. Lemma~\ref{lem:quartic-refinement} and \eqref{eq:adjacency} give $N\le10$. For $f=4$, choose three facets with $e(T)\le3+\rho_F$. If $\rho_F=1$, take three consecutive facets; if $\rho_F=2$, take a consecutive pair and a third facet; if $\rho_F\ge3$, any triple suffices. The same inequalities give $N\le11$.

If $f\le1$, all vertices are rank one, so $N\le12$. The same holds if $f=2$ and the two facets are nonadjacent. If they are adjacent, $\nu_2\le1$. For a singular integral quartic, $\nu_1\le10$ gives $N\le11$. For a smooth quartic the two facets have three incident vertices, and the second assertion of Lemma~\ref{lem:quartic-refinement} gives $\nu_1\le11$, hence $N\le12$.
\end{proof}

\begin{proof}[Proof of Theorem~\ref{cor:rank-four}]
Let $P$ be an $N$-gon with a size-four PSD factorization. We may assume
$N\ge6$ and, by polarity, $u\le w$. Equation~\eqref{eq:factor-spans}
then gives $u\le6$.
If $u\le4$, the compact strictly feasible slice $K_U$ has affine dimension
$u-1\le3$. For a three-dimensional slice, a generic planar shadow has
type $(4,2,1)$: matrix order four, image dimension two, and one eliminated
variable. In \cite[Theorem 1.1 and Table 1]{SinnSturmfels2015}, the entry
$p=1$, $n=4$ has only the rank-three branch, of degree
$\delta(2,4,3)=12$. Thus its algebraic boundary has degree twelve.
For a two-dimensional slice there are no eliminated variables, and
the determinant bound is four.
The degree bound persists under degeneration. To see this, write the
slice in parameter form as
\[
 K_0=\{t:B_0+\mathcal A_0(t)\succeq0\},
 \qquad B_0\succ0,
\]
with $\mathcal A_0$ injective. Take generic perturbations
$B_\epsilon\to B_0$, $\mathcal A_\epsilon\to\mathcal A_0$ and put
$K_\epsilon=\{t:B_\epsilon+\mathcal A_\epsilon(t)\succeq0\}$.
Compactness of $K_0$ gives a uniform $\eta>0$ such that
$\lambda_{\min}(\mathcal A_0(t))\le-\eta$ on the unit sphere.
For small $\epsilon$ this remains at most $-\eta/2$, so every
$t\in K_\epsilon$ satisfies $\|t\|\le2\|B_\epsilon\|/\eta$.
Thus the perturbed sets are uniformly bounded, and strict feasibility
persists.

For $t\in K_0$ and fixed $0<\delta<1$, one has
\[
 B_0+\mathcal A_0((1-\delta)t)
 =(1-\delta)(B_0+\mathcal A_0(t))+\delta B_0
 \succeq\delta\lambda_{\min}(B_0)I.
\]
The positive margin is uniform on $K_0$, so
$(1-\delta)K_0\subseteq K_\epsilon$ for sufficiently small $\epsilon$.
Conversely, every limit of feasible parameters from $K_\epsilon$ lies
in $K_0$, by continuity of the pencil and closedness of the PSD cone.
Uniform boundedness and then $\delta\downarrow0$ give Hausdorff
convergence. Choose generic projections converging to the original
one; their planar images converge as well.

Normalize the resulting boundary polynomials to coefficient norm one
and take a convergent subsequence. The limit is a nonzero polynomial
of degree at most twelve. Boundaries of full-dimensional compact convex
bodies converge under this Hausdorff limit, so the polynomial vanishes
on every polygon edge. Its restriction to each edge line vanishes on
a nonempty interval, hence identically. The distinct edge lines
therefore divide it, giving $N\le12$.

For the remaining dimensions, Corollary~\ref{cor:five-dim-facets}, Lemma~\ref{lem:corner-loci}, and the run bound \eqref{eq:factor-runs} give the following preliminary bounds. The case $(u,w)=(5,8)$ uses Corollary~\ref{cor:five-eight}.
\begin{center}
\begin{tabular}{@{}ccccc@{}}
\toprule
$u$ & $w$ & $f_1$ & $\nu_1$ & Preliminary bound \\
\midrule
$\le4$ & $u\le w\le13-u$ & --- & --- & $N\le12$ \\
$5$ & $5\le w\le7$ & $\le4$ & $\le8$ & $N\le11$ \\
$5$ & $8$ & $\le4$ & $\le12$ & $N\le12$ \\
$6$ & $6$ & $\le6$ & $\le6$ & $N\le11$ \\
$6$ & $7$ & $\le6$ & $\le8$ & $N\le13$ \\
\bottomrule
\end{tabular}
\end{center}
If one side is entirely rank one, $N$ equals its rank-one count and gives an even smaller bound, except in the already settled $(5,8)$ case. These possibilities exhaust $u\le w$ and $u+w\le13$. Thus thirteen vertices could occur only when $(u,w)=(6,7)$. Proposition~\ref{prop:exclude-thirteen} below excludes that configuration, proving $N\le12$.

The exact octagon factorization in \cite[Example 5.2]{VGG2018} gives
$V_2(4)\ge8$. By \cite[Proposition 3.2]{GRT2013}, a polytope of PSD
rank at most four has dimension at most three. The three-dimensional
classification \cite[Theorem 4.11]{GRT2013} gives at most eight vertices
in dimension three; segments have two. Hence $V(4)=V_2(4)$.
\end{proof}

\subsection{Excluding the last configuration}\label{sec:exclude-thirteen}
The equality case in the run bound fixes the cyclic ranks. Restricting the six rank-one facet forms to the kernel of one rank-two facet gives six distinct points on a projective line. Cubic interpolation then places the facet vectors on a moment curve, and the vertex factors become cubics whose consecutive pairs share two roots. Their squares span a four-dimensional space. Three squares vanishing at one point force every element of that space which vanishes there to have a double zero. A suitable difference of two other squares has a simple zero, giving the contradiction.

\begin{proposition}\label{prop:exclude-thirteen}
A thirteen-gon has no size-four PSD slack factorization with $(u,w)=(6,7)$.
\end{proposition}
\begin{proof}
Suppose such a factorization exists, with $A_i$ indexed by the edge $[v_i,v_{i+1}]$ and $B_i$ by $v_i$, cyclically modulo thirteen. The rank-one counts are at most six and eight. Both sides contain rank-two factors, and equality is forced throughout
\[
 13\le f_1+\nu_1-\max(\rho_F,\rho_V)\le6+8-1=13.
\]
After cyclic relabelling, $A_0,\ldots,A_6$ have rank two, $A_7,\ldots,A_{12}$ have rank one, and $B_0,\ldots,B_7$ have rank one. Write
\[
 A_{6+i}=a_i a_i^{\mathsf T}\quad(1\le i\le6),\qquad
 B_j=q_jq_j^{\mathsf T}\quad(0\le j\le7).
\]
By Theorem~\ref{lem:corner-independent}, the six facet projectors form a basis of $U$, and any seven of the eight vertex projectors form a basis of $W$. The rank-three trace pairing and $u+w=13$ give
\begin{equation}\label{eq:thirteen-complements}
 U^\perp\subset W.
\end{equation}
Lemma~\ref{lem:five-factors} shows that every rank-two facet has a full-support expansion in the six rank-one facet factors, with four positive and two negative coefficients.

The parameter line comes from the kernel of a rank-two facet. Fix such a factor $A=FDF^{\mathsf T}$, where $F=[a_1\ \cdots\ a_6]$ and $D=\operatorname{diag}(d_1,\ldots,d_6)$ is invertible. The matrix $F$ has row rank four because $U$ contains a positive definite matrix. The six forms $a_i^{\mathsf T}|_{\ker A}$ are nonzero and pairwise nonproportional. A zero restriction would make its facet contain both endpoints of the different facet of $A$. Two proportional restrictions would make the corresponding slacks proportional at both endpoints, giving a linear dependence with the slack of $A$. Both conclusions contradict polygon incidence: three distinct facet slacks are independent.

Choose a basis matrix $K$ of $\ker A$ so that
\[
 d_i a_i^{\mathsf T}K=c_i(1,t_i),\qquad c_i\ne0,
\]
with the six real $t_i$ distinct. Since $F$ has row rank four and $AK=0$, the two columns of $DF^{\mathsf T}K$ form a basis of $\ker F$. Thus this kernel is spanned by $(c_i)_{i=1}^6$ and $(c_it_i)_{i=1}^6$. Put
\[
 v(t)=(1,t,t^2,t^3)^{\mathsf T},\qquad p(t)=\prod_{i=1}^6(t-t_i).
\]
Lagrange interpolation gives $\sum_i t_i^j/p'(t_i)=0$ for $0\le j\le4$. Hence, for $\mathsf V=[v(t_1)\ \cdots\ v(t_6)]$, the kernel of $\mathsf V$ is spanned by $(1/p'(t_i))_i$ and $(t_i/p'(t_i))_i$. Comparing the two kernels gives
\[
 F=T \mathsf V\operatorname{diag}\!\left(\frac1{c_i p'(t_i)}\right)
\]
for an invertible real $T$. Apply congruence by $T^{-1}$ to the facet factors and the inverse congruence to the vertex factors. Thus $U$ lies in the seven-dimensional Hankel space $\mathcal H$, spanned by $v(t)v(t)^{\mathsf T}$, with $v(\infty)=(0,0,0,1)^{\mathsf T}$.

In these coordinates, positive semidefiniteness gives a real moment representation. Write a Hankel matrix as $A_{ij}=m_{i+j}$ and associate to it the functional $L_A(t^k)=m_k$ on binary sextics, written in affine coordinates. For a cubic $h$ with coefficient vector $c$, one has $L_A(h^2)=c^{\mathsf T}Ac$. Every nonnegative real binary sextic is a sum of two squares of real binary cubics: factoring over $\mathbb C$ writes it as $g\overline g=(\Re g)^2+(\Im g)^2$. Thus $A\succeq0$ makes $L_A$ nonnegative on every nonnegative sextic. The evaluation matrices $v(t)v(t)^{\mathsf T}$, normalized to trace one and including $t=\infty$, form a compact set. Their convex hull is a compact base in the trace-one hyperplane, so the convex cone they generate is closed. Duality therefore expresses $A$ as a positive combination of these evaluation matrices.

Consequently each rank-two facet has a representation
\begin{equation}\label{eq:two-moment-atoms}
 A_i=\alpha_i v(s_i)v(s_i)^{\mathsf T}+\beta_i v(r_i)v(r_i)^{\mathsf T},\qquad
 \alpha_i,\beta_i>0,\quad s_i\ne r_i,
 \quad 0\le i\le6.
\end{equation}
All atoms lie in the two-dimensional range of $A_i$, and any three distinct points of the cubic moment curve are independent, so precisely two distinct atoms occur. Neither is one of the $t_k$, since that rank-one facet would then contain both endpoints of the facet of $A_i$. The seven unordered pairs are distinct: equal pairs would give equal kernels, making two distinct facets contain the same two endpoints.

Define the real binary cubics $h_j(t)=q_j^{\mathsf T}v(t)$, $0\le j\le7$. Orthogonality at the endpoints of $A_i$ shows that $h_i$ and $h_{i+1}$ both vanish at $s_i,r_i$. For $1\le j\le6$, the two distinct pairs coming from $A_{j-1},A_j$ lie among the roots of $h_j$; they overlap in exactly one point and give three distinct real roots. The endpoint cubics $h_0,h_7$ have, respectively, the additional roots $t_6,t_1$ from their incident rank-one facets. Thus all eight cubics have three distinct real projective roots, and no two are proportional.

Writing $v(s,t)=(s^3,s^2t,st^2,t^3)^{\mathsf T}$ for the homogeneous version of $v(t)=v(1,t)$, consider multiplication
\[
 \mu:\mathbb S^4\longrightarrow\R[s,t]_6,\qquad B\longmapsto v(s,t)^{\mathsf T}Bv(s,t).
\]
Its kernel is $\mathcal H^\perp$, of dimension three. By \eqref{eq:thirteen-complements}, $\mathcal H^\perp\subset U^\perp\subset W$, so
\begin{equation}\label{eq:four-sextics}
 \mathcal L=\mu(W)=\operatorname{span}\{h_0^2,\ldots,h_7^2\},\qquad \dim\mathcal L=4.
\end{equation}
This space has no base point over $\mathbb C$: the $q_j$ span $\R^4$, since their projectors span $W$ and $W$ contains a positive definite matrix. Any three distinct squares are independent. A dependence would express one square as a nonnegative combination of at most two others; evaluation at its three distinct real roots would make the other cubics proportional to it.

At most four cubics can vanish at a real projective point $a$. Indeed, cubics vanishing at $a$ form a three-dimensional space $E_a$. Multiplication from $\operatorname{Sym}^2E_a$ to sextics has a one-dimensional kernel, as is seen by removing the common squared linear factor and multiplying binary quadratics. Five vanishing cubics would give five independent projectors, hence at least four independent squares. But the sections of $\mathcal L$ vanishing at $a$ form a space of dimension three, a contradiction.

Let $a$ be the common root of $h_0,h_1,h_2$ coming from the node pairs of $A_0,A_1$, and choose an affine parameter in which $a$ is finite. Since $\mathcal L$ has dimension four and no base point, its subspace of elements vanishing at $a$ has dimension three. The independent squares $h_0^2,h_1^2,h_2^2$ form a basis of that subspace. Each has a double zero at $a$, so every element of $\mathcal L$ that vanishes there has a double zero.

At most four of the eight cubics vanish at $a$, so at least four among $h_3,\ldots,h_7$ do not. Two of these are consecutive, say $h_j,h_{j+1}$. Write $h_j=Q\ell$ and $h_{j+1}=Qm$, with $Q$ quadratic and $\ell,m$ nonproportional linear forms. All three are nonzero at $a$. Now
\[
 m(a)^2h_j^2-\ell(a)^2h_{j+1}^2
   =Q^2\bigl(m(a)\ell-\ell(a)m\bigr)
        \bigl(m(a)\ell+\ell(a)m\bigr)
\]
belongs to $\mathcal L$ and vanishes at $a$. On the right, the first linear factor is nonzero as a polynomial and vanishes at $a$, whereas the other factors do not vanish there. This gives a simple zero at $a$, contradicting the double zero forced by the preceding paragraph.
\end{proof}

\section{Geometric rank-four lifts for square-symmetric octagons}\label{sec:octagons}
Vandaele--Glineur--Gillis \cite[Example 5.2]{VGG2018} already gave an
explicit exact size-four PSD factorization of the regular octagon.
Here we give a geometric construction for a family of octagons and,
by polarity, for every octagon with the standard square symmetries.
The auxiliary polytope is a biplanar octahedron, whose PSD rank is known
from \cite[Theorem 4.8]{GRT2013}. The direct calculation supplies an
explicit lift of its section and recovers the known regular case.

\begin{theorem}\label{thm:octagons}
Let $h>1$ and $\varrho=h/(h+1)$. The octagon
\begin{equation}\label{eq:oct-family}
P_h=\conv\{(\pm \varrho,0),(0,\pm \varrho),(\pm\tfrac12,\pm\tfrac12)\}
\end{equation}
has real PSD rank four. All four sign choices are included in the last pair.
The regular case $h=1+\sqrt2$ recovers the rank-four result of
\cite[Example 5.2]{VGG2018}.
\end{theorem}
Figure~\ref{fig:octagon-section} shows the section underlying the construction.

\begin{figure}[tbp]
\centering
\vspace*{12pt}
\begin{tikzpicture}[font=\small,line join=round]
\begin{scope}[x={(1.65cm,-.45cm)},y={(1.05cm,.70cm)},z={(0cm,.90cm)}]
\coordinate (up) at (1,0,1); \coordinate (um) at (-1,0,1);
\coordinate (vp) at (0,1,-1); \coordinate (vm) at (0,-1,-1);
\coordinate (wp) at (0,0,2.41421356); \coordinate (wm) at (0,0,-2.41421356);
\fill[figteal!7] (-1.05,-1.05,0)--(1.05,-1.05,0)--(1.05,1.05,0)--(-1.05,1.05,0)--cycle;
\draw[figteal!35,densely dashed] (-1.05,-1.05,0)--(1.05,-1.05,0)--(1.05,1.05,0)--(-1.05,1.05,0)--cycle;
\foreach \a in {up,um} \foreach \b in {vp,vm,wp,wm}
 {\draw[black!45,line width=.55pt] (\a)--(\b);}
\foreach \a in {vp,vm} \foreach \b in {wp,wm}
 {\draw[black!45,line width=.55pt] (\a)--(\b);}
\filldraw[fill=figteal!20,draw=figteal,line width=1.2pt, opacity = 0.7]
 (.70710678,0,0)--(.5,.5,0)--(0,.70710678,0)--(-.5,.5,0)--
 (-.70710678,0,0)--(-.5,-.5,0)--(0,-.70710678,0)--(.5,-.5,0)--cycle;
\foreach \a/\pos/\lab in {up/right/u_+,um/left/u_-,vp/right/v_+,vm/left/v_-,wp/above/w_+,wm/below/w_-}
 {\fill[figblue] (\a) circle (1.5pt);\node[\pos] at (\a) {$\lab$};}
\node[figteal,below right] at (1.05,-1.05,0) {$z=0$};
\end{scope}
\draw[-{Stealth[length=2mm]},thick,figteal] (3.1,0)--node[above,font=\footnotesize]{section} (4.3,0);
\begin{scope}[shift={(6.55,0)},scale=2.2]
\draw[black!30,-{Stealth[length=1.5mm]}] (-.95,0)--(.98,0) node[right] {$x$};
\draw[black!30,-{Stealth[length=1.5mm]}] (0,-.95)--(0,.98) node[above] {$y$};
\filldraw[fill=figteal!12,draw=figteal,line width=1.2pt]
 (.70710678,0)--(.5,.5)--(0,.70710678)--(-.5,.5)--
 (-.70710678,0)--(-.5,-.5)--(0,-.70710678)--(.5,-.5)--cycle;
\foreach \x/\y in {.70710678/0,.5/.5,0/.70710678,-.5/.5,-.70710678/0,-.5/-.5,0/-.70710678,.5/-.5}
 {\fill[figteal] (\x,\y) circle (.7pt);}
\node[below right,font=\footnotesize] at (.70710678,0) {$(\varrho,0)$};
\node[above right,font=\footnotesize] at (.5,.5) {$(\tfrac12,\tfrac12)$};
\end{scope}
\node at (0,-3) {$Q_h\subset\mathbb R^3$};
\node at (6.55,-3) {$P_h=Q_h\cap\{z=0\}$};
\end{tikzpicture}
\caption{The octagon construction in Theorem~\ref{thm:octagons}, shown
at $h=1+\sqrt2$ and $\varrho=h/(h+1)=1/\sqrt2$ (the regular case).
The six vertices of $Q_h$ lie above or below the section plane.
Four edges $u_sv_t$ give $(s/2,t/2)$; the edges $u_sw_-$ and $v_tw_+$
give the four axial vertices. Thus a plane cuts the size-four lifted
octahedron in an octagon. The same construction works for every $h>1$.}
\label{fig:octagon-section}
\end{figure}
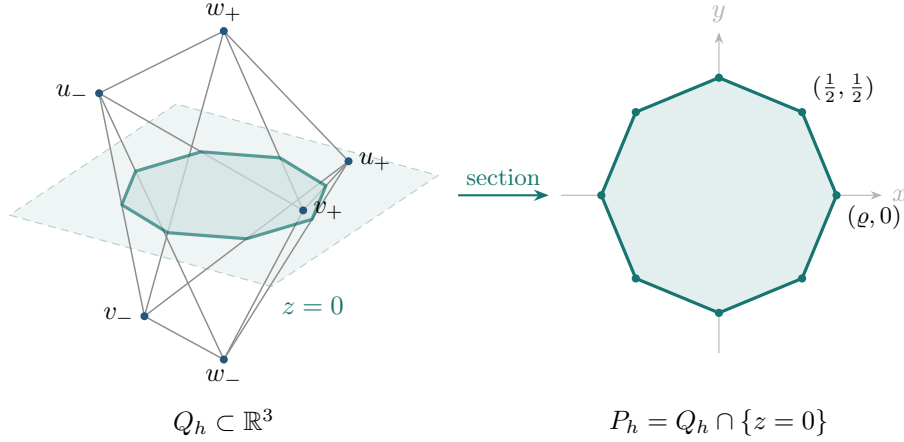
\begin{proof}
Consider
\[
Q_h=\conv\{u_+,u_-,v_+,v_-,w_+,w_-\},\quad
u_\pm=(\pm1,0,1),\quad v_\pm=(0,\pm1,-1),\quad w_\pm=(0,0,\pm h).
\]
For $s,t,\varepsilon\in\{\pm1\}$ the eight facet inequalities are
\begin{equation}\label{eq:oct-facets}
s(1-\varepsilon/h)x+t(1+\varepsilon/h)y+(\varepsilon/h)z\le1.
\end{equation}
Each is tight at exactly $u_s,v_t,w_\varepsilon$ and strict at the other
three vertices. These eight distinct supporting triangles are all the
facets: a three-dimensional polytope with six vertices has at most eight
facets, by Euler's formula and the fact that every facet has at least three
edges. In particular, $Q_h$ is a combinatorial octahedron.

Set
\[
a_\varepsilon=\sqrt{2(1-\varepsilon/h)},\qquad
b_\varepsilon=\sqrt{2(1+\varepsilon/h)}.
\]
In the column order $u_+,u_-,v_+,v_-,w_+,w_-$, the nonnegative square root of
the slack matrix has row
\begin{equation}\label{eq:root-row}
R_{s,t,\varepsilon}=\left(
\frac{1-s}{2}a_\varepsilon,\frac{1+s}{2}a_\varepsilon,
\frac{1-t}{2}b_\varepsilon,\frac{1+t}{2}b_\varepsilon,
\frac{1-\varepsilon}{\sqrt2},\frac{1+\varepsilon}{\sqrt2}\right).
\end{equation}
Both $a_\varepsilon$ and $b_\varepsilon$ are linear combinations of $1$ and
$\varepsilon$, since $\varepsilon$ takes two values. The six columns
therefore lie in the span of the four functions
\[
1,\quad\varepsilon,\quad s a_\varepsilon,\quad t b_\varepsilon.
\]
Thus $\rank R\le4$, and $\xcPSD(Q_h)=4$ by Lemma~\ref{lem:basic}.

The plane $z=0$ meets the edges $u_sv_t$, $u_sw_-$ and $v_tw_+$ in,
respectively,
\[
(s/2,t/2,0),\qquad(s\varrho,0,0),\qquad(0,t\varrho,0).
\]
These are all the edge intersections with that plane, so
$Q_h\cap\{z=0\}$ is $P_h$. Since $1/2<\varrho<1$, all eight displayed points are
vertices. Lemma~\ref{lem:faces} gives $\xcPSD(P_h)\le4$; the classification
of polygons of PSD rank three gives the reverse inequality
\cite[Theorem 4.7]{GRT2013}.

For $h=1+\sqrt2$ one has $\varrho=1/\sqrt2$. Multiplying $P_h$ by $\sqrt2$ gives
the octagon with vertices $(\pm1,0)$, $(0,\pm1)$ and
$(\pm1/\sqrt2,\pm1/\sqrt2)$, which is regular.
\end{proof}

\begin{remark}\label{rem:explicit-lift}
Define $a=\sqrt{2(1-1/h)}$, $b=\sqrt{2(1+1/h)}$,
$\alpha=(a+b)/2$, $\beta=(a-b)/2$, and
\[
q_{s,t,\varepsilon}=(1,\varepsilon,s a_\varepsilon,t b_\varepsilon)^{\mathsf T}.
\]
The six coefficient vectors in the factorization
$R_{ij}=q_i^{\mathsf T}c_j$ are
\begin{align*}
c_{u_\pm}&=\tfrac12(\alpha,\beta,\mp1,0)^{\mathsf T},&
c_{v_\pm}&=\tfrac12(\alpha,-\beta,0,\mp1)^{\mathsf T},\\
c_{w_\pm}&=\tfrac1{\sqrt2}(1,\mp1,0,0)^{\mathsf T}.&&
\end{align*}
In particular the polygon in \eqref{eq:oct-family} has the lift
\begin{equation}\label{eq:explicit-lift}
P_h=\left\{(x,y):\begin{array}{l}
\text{there exists }X\in\PSD4\text{ such that, for all }s,t,\varepsilon\in\{\pm1\},\\[2pt]
q_{s,t,\varepsilon}^{\mathsf T}Xq_{s,t,\varepsilon}
=1-s(1-\varepsilon/h)x-t(1+\varepsilon/h)y
\end{array}\right\}.
\end{equation}
For necessity, nonnegativity of the left side gives the inequalities of the
section. For sufficiency, each vertex of $Q_h$ lifts to $c_jc_j^{\mathsf T}$;
convex combinations lift every point of $Q_h$, including its section.
The polygon's vertex factors obtained this way may have rank two, because a
section vertex lies in an edge of $Q_h$. The construction does not assert a
rank-four square root of the octagon's own slack matrix.
\end{remark}

The eight facet factors also realize the first dependence in
Theorem~\ref{lem:corner-independent}. Directly from their coordinates,
\begin{equation}\label{eq:octagon-two-bases}
 \sum_{st\varepsilon=1}q_{s,t,\varepsilon}q_{s,t,\varepsilon}^{\mathsf T}
 =\sum_{st\varepsilon=-1}q_{s,t,\varepsilon}q_{s,t,\varepsilon}^{\mathsf T}
 =\operatorname{diag}(4,4,8,8).
\end{equation}
Their images are the eight vertices of the polar octagon, so any seven
are independent. Whitening the common positive definite sum turns the
two groups of four vectors into orthonormal bases. Consequently,
$\dim U=7$ and $c(U)=8$, so the same construction attains the corner
bound in Lemma~\ref{lem:corner-loci} and realizes the first dependence
allowed by Theorem~\ref{lem:corner-independent}.

\begin{lemma}
\label{lem:rank-three-lines}
Let
\[
 X(x,y,\tau)=C+xD_x+yD_y+\tau_1K_1+\tau_2K_2+\tau_3K_3,
 \qquad D_x,D_y,K_j,C\in\mathbb S^4,
\]
be a strictly feasible pencil with compact spectrahedron
$Q\subseteq\mathbb R^5$ and planar shadow $P=\pi(Q)$, where
$\pi(x,y,\tau)=(x,y)$. Define
\[
 Z=\{[z]\in\mathbb P^3_{\mathbb C}:z^{\mathsf T}K_jz=0, j=1,2,3\},
 \qquad
 Z_{\mathrm{act}}=\{[z]\in Z:
 (z^{\mathsf T}D_xz,z^{\mathsf T}D_yz)\ne(0,0)\}.
\]
If $Z$ is finite, then $|Z|\le8$. The critical values of $\pi$ on the
smooth rank-three locus of $\{\det X=0\}\subseteq\mathbb C^5$ lie on
the at most eight lines
\begin{equation}\label{eq:branch-line}
 L_z:\quad z^{\mathsf T}Cz+xz^{\mathsf T}D_xz+yz^{\mathsf T}D_yz=0,
 \qquad [z]\in Z_{\mathrm{act}}.
\end{equation}
If, in addition, $P$ is a polygon and the relative interior of every edge contains
an open segment of points admitting rank-three PSD lifts, then $P$
has at most eight edges.
\end{lemma}
\begin{proof}
Put $F=\det X$. At a rank-three point with kernel $\mathbb C z$,
$\operatorname{adj}(X)=\gamma zz^{\mathsf T}$ for some $\gamma\ne0$, so
\[
 dF=\gamma\bigl(z^{\mathsf T}D_xz,z^{\mathsf T}D_yz,
 z^{\mathsf T}K_1z,z^{\mathsf T}K_2z,z^{\mathsf T}K_3z\bigr).
\]
At a smooth point, the projection is critical precisely when this
tangent hyperplane contains the three vertical coordinate directions.
Thus $[z]\in Z$, and smoothness forces $[z]\in Z_{\mathrm{act}}$.
The identity $Xz=0$ then gives \eqref{eq:branch-line}, which is a genuine
line because its two variable coefficients do not both vanish.
Three quadrics in $\mathbb P^3$ with finite common zero locus have
at most eight distinct common zeros by B\'ezout's theorem.

For the polygon assertion, any rank-three PSD lift of a boundary point
is smooth in $\{F=0\}$: the direction toward a strictly feasible matrix
has nonzero determinant differential, since its quadratic form is
positive on a real kernel vector. Locally the smooth hypersurface is
PSD, as its other three eigenvalues remain positive. Its projection
cannot be a submersion at a boundary point of $P$, since that would
put an open planar neighborhood in $P$. Such a lift is therefore
critical. Every edge line contains infinitely many critical values
and hence coincides with one of the finitely many lines $L_z$.
Distinct edges have distinct supporting lines, proving the bound.
\end{proof}

\begin{remark}\label{rem:branch-vs-loci}
The restriction to $Z_{\mathrm{act}}$ excludes identically zero equations
that cannot arise from smooth critical points. Strict feasibility alone
does not exclude them over $\mathbb C$: a nonzero complex vector can
satisfy $z^{\mathsf T}Hz=0$ even for $H\succ0$.
For generic data of type $(4,2,3)$, the degree-eight rank-three branch
splits into eight lines over $\mathbb C$. Irreducibility over the field
generated by the pencil coefficients is compatible with this splitting;
see \cite[Theorem 1.1 and Remark 2.2]{SinnSturmfels2015}.

Choose the polarity in Proposition~\ref{prop:p3-spans} so that $\dim W\le6$.
For the resulting span-reduced lift, the vertical matrix space is
\[
 \mathcal V=W\cap U^\perp,\qquad \dim \mathcal V=\dim W-3\le3.
\]
Its kernel locus satisfies
\[
 Z_{\mathcal V}=\{[z]\in\mathbb P^3_{\mathbb C}:zz^{\mathsf T}\perp \mathcal V_{\mathbb C}\}
 \supseteq X_U,
 \qquad X_U=\{[z]\in\mathbb P^3_{\mathbb C}:zz^{\mathsf T}\in U_{\mathbb C}\}.
\]
The inclusion follows from $\mathcal V\subseteq U^\perp$. Equality holds when
$\mathcal V=U^\perp$; in general neither equality nor finiteness follows from
this inclusion. There are three independent defining quadrics only
when $\dim W=6$. If $\dim W<6$, the locus $Z_{\mathcal V}$ is necessarily
positive-dimensional. Thus the full pencil and the span-reduced lift
must be distinguished when applying Lemma~\ref{lem:rank-three-lines}.
\end{remark}

\begin{proposition}\label{prop:octagon-branch}
For every $h>1$, the lift in \eqref{eq:explicit-lift} has type $(4,2,3)$.
Each point in the relative interior of an edge has a unique lift,
of rank three; each vertex has a unique lift, of rank two.
The reduced Zariski closure of the critical values of the projection
on the smooth rank-three determinantal locus is the union of the
eight edge lines, with defining polynomial
\[
 b_h(x,y)=\prod_{s,t,\varepsilon\in\{\pm1\}}
 \bigl(1-s(1-\varepsilon/h)x-t(1+\varepsilon/h)y\bigr).
\]
Its degree is eight. These eight rank-three branch lines persist
under sufficiently small generic perturbations and specialize to $b_h$.
\end{proposition}
\begin{proof}
With $\alpha,\beta$ as in Remark~\ref{rem:explicit-lift}, the equations
of that lift are equivalent to $X\succeq0$ for the pencil
\[
 X(x,y,\tau)=
 \begin{pmatrix}
 1-\tau_1-2\tau_2-2\tau_3&(\tau_2-\tau_3)/h&-\alpha x/4&-\alpha y/4\\
 (\tau_2-\tau_3)/h&\tau_1&-\beta x/4&\beta y/4\\
 -\alpha x/4&-\beta x/4&\tau_2&0\\
 -\alpha y/4&\beta y/4&0&\tau_3
 \end{pmatrix}.
\]
Indeed, expansion of $q_{s,t,\varepsilon}^{\mathsf T}Xq_{s,t,\varepsilon}$
gives the required slack. Conversely, comparing the seven independent
sign functions $1,\varepsilon,sa_\varepsilon,\varepsilon sa_\varepsilon,
tb_\varepsilon,\varepsilon tb_\varepsilon,st$ gives exactly these equations.
The five pencil directions are independent. The average of the six
matrices $c_jc_j^{\mathsf T}$ is a positive-definite feasible point
at $(x,y)=(0,0)$, and
$\tr(\operatorname{diag}(1,1,2,2)X)=1$ proves compactness.
Thus the affine dimension is five and the projection eliminates three variables.

Write $K_i=\partial X/\partial\tau_i$. The kernel locus $Z$ of
Lemma~\ref{lem:rank-three-lines} is defined here by
\begin{equation}\label{eq:octagon-kernel-quadrics}
 z_2^2-z_1^2=0,\qquad
 z_3^2-2z_1^2+(2/h)z_1z_2=0,\qquad
 z_4^2-2z_1^2-(2/h)z_1z_2=0.
\end{equation}
There are exactly eight projective solutions, namely
$[q_{s,t,\varepsilon}]$. They are simple: $z_1\ne0$, and in the
chart $z_1=1$ the Jacobian determinant is $8z_2z_3z_4\ne0$.
For each solution, \eqref{eq:branch-line} is exactly the corresponding
edge line, so $Z_{\mathrm{act}}=Z$. Lemma~\ref{lem:rank-three-lines}
places all smooth rank-three critical values on these eight lines.

Conversely, the edge indexed by $(s,t,\varepsilon)$ is the section
of the triangular facet $\conv\{u_s,v_t,w_\varepsilon\}$ of $Q_h$.
An interior point of the section has strictly positive barycentric
weights on these three vertices. Its lift is therefore
\[
 X=\lambda_u c_{u_s}c_{u_s}^{\mathsf T}
   +\lambda_v c_{v_t}c_{v_t}^{\mathsf T}
   +\lambda_w c_{w_\varepsilon}c_{w_\varepsilon}^{\mathsf T},
 \qquad \lambda_u,\lambda_v,\lambda_w>0.
\]
The three vectors are independent: the minor on rows $1,3,4$ of
$[c_{u_s}\ c_{v_t}\ c_{w_\varepsilon}]$ equals $st/(4\sqrt2)$.
Thus $X$ has rank three and kernel $q_{s,t,\varepsilon}$.
The differential in $(x,y)$ is a nonzero multiple of the edge normal,
so this is a smooth critical point. At an endpoint exactly two
weights are positive, giving rank two.

Every PSD lift over this edge satisfies $Xq=0$, since $q^{\mathsf T}Xq=0$.
The vectors $K_1q,K_2q,K_3q$ are independent by the same Jacobian
calculation, so $Xq=0$ determines the three fiber coordinates uniquely.
This proves both uniqueness assertions. Each edge supplies an open
segment of critical values, so the reduced branch polynomial is exactly
$b_h$ and the eight-line bound of Lemma~\ref{lem:rank-three-lines} is attained.

Finally, the eight simple projective kernel solutions continue under
small perturbations of the three vertical matrices. Each continued
kernel gives a line $q^{\mathsf T}X(x,y,\tau)q=0$, independent of $\tau$;
the equations $Xq=0$ still determine $\tau$ uniquely along that line.
Near an interior edge point, positivity on $q^\perp$ persists.
The eight lines therefore remain rank-three boundary branches, and
their product specializes, up to a nonzero scalar, to $b_h$.
\end{proof}
Figure~\ref{fig:edge-ranks} shows why the edge interiors have rank three
although their endpoints have rank two.

\begin{figure}[tbp]
\centering
\vspace*{12pt}
\begin{tikzpicture}[font=\small,line join=round]
\coordinate (u) at (0,0);\coordinate (v) at (4,0);\coordinate (w) at (2,3);
\coordinate (d) at (2,0);\coordinate (a) at (3.41421356,.87867966);
\filldraw[fill=figblue!6,draw=figblue!60] (u)--(v)--(w)--cycle;
\draw[figteal,line width=1.7pt] (d)--(a);
\fill[figblue] (u) circle (1.5pt);\fill[figblue] (v) circle (1.5pt);\fill[figblue] (w) circle (1.5pt);
\node[below left,align=center] at (u) {$u_s$\\\footnotesize rank $1$};
\node[below right,align=center] at (v) {$v_t$\\\footnotesize rank $1$};
\node[above,align=center] at (w) {$w_+$\\\footnotesize rank $1$};
\fill[figorange] (d) circle (2.1pt);\fill[figorange] (a) circle (2.1pt);
\node[below,figorange,font=\footnotesize] at (d) {rank $2$};
\node[right,figorange,font=\footnotesize] at (a) {rank $2$};
\node[above left,figteal,font=\footnotesize] at ($(d)!.55!(a)$) {rank $3$};
\draw[-{Stealth[length=2mm]},thick,figteal] (4.65,1.1)--node[above,font=\footnotesize]{section edge} (6.25,1.1);
\begin{scope}[shift={(8.1,1.1)},scale=1.8]
\filldraw[fill=figteal!5,draw=black!45]
 (.70710678,0)--(.5,.5)--(0,.70710678)--(-.5,.5)--
 (-.70710678,0)--(-.5,-.5)--(0,-.70710678)--(.5,-.5)--cycle;
\draw[figteal,line width=1.7pt] (.5,.5)--(0,.70710678);
\fill[figorange] (.5,.5) circle (1.2pt);\fill[figorange] (0,.70710678) circle (1.2pt);
\node[above right,font=\footnotesize] at (.5,.5) {$(s/2,t/2)$};
\node[above left,font=\footnotesize] at (0,.70710678) {$(0,t\varrho)$};
\end{scope}
\end{tikzpicture}
\caption{Ranks along a section edge, shown for $h=1+\sqrt2$ and $s=t=\varepsilon=1$.
The left triangle is a barycentric view of the facet
$\conv\{u_s,v_t,w_+\}$; its teal segment is the cut $z=0$.
The three corners lift to rank-one matrices. Each endpoint of the cut
mixes two independent vectors and has rank two, while an interior
point mixes all three and has rank three. The right panel locates
that edge in the octagon. These are matrix ranks in the explicit PSD
lift, not dimensions of faces of $Q_h$.}
\label{fig:edge-ranks}
\end{figure}
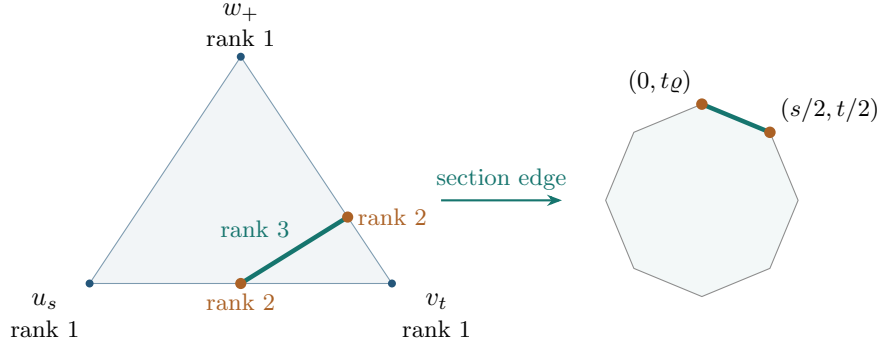

The proposition identifies the degree-eight branch for this particular
lift. It does not show that arbitrary size-four polygon lifts have
rank-three edge strata. In \cite[Example 4.2]{SinnSturmfels2015}, both
rank-two and rank-three branch loci contain the edge lines; in their
Example~4.3 only rank two is ramified. Thus branch selection requires
an argument about the representation, not only the polygon.

\begin{corollary}\label{cor:square-symmetry}
Every convex octagon invariant under $(x,y)\mapsto(-y,x)$ and
$(x,y)\mapsto(x,-y)$ has real PSD rank four.
\end{corollary}
\begin{proof}
The center is the origin. A nonzero orbit under this group has four points
if it lies on the axes or the diagonals, and eight points otherwise.
There are two possibilities for an eight-vertex invariant polygon.

If the vertex set consists of two four-point orbits, one must lie on the
axes and the other on the diagonals: two distinct orbits on the same four
rays cannot both consist of extreme points. Write these vertices as
$(\pm A,0),(0,\pm A),(\pm B,\pm B)$ with $A,B>0$.
All eight are extreme precisely when $B<A<2B$. Scaling by $1/(2B)$ gives
\eqref{eq:oct-family}, with $\varrho=A/(2B)$ and $h=\varrho/(1-\varrho)>1$.

Otherwise the vertices form the orbit of $(a,b)$ with $a>b>0$.
The eight edge normals of this polygon lie alternately on the axes and
the diagonals. Its polar is therefore an octagon of the first type.
Explicitly, its polar has axis vertices $(\pm1/a,0),(0,\pm1/a)$ and diagonal
vertices $(\pm1/(a+b),\pm1/(a+b))$; here $1<(a+b)/a<2$.
Polarity preserves PSD rank, proving the claim.
\end{proof}

The regular member also distinguishes symmetry of a polygon from symmetry
of a smallest lift. Following \cite{FSP2015equivariant,FSP2017polygons},
a lift is \emph{rotation-equivariant} if there is a homomorphism
$\rho:C_8\to\mathrm{GL}_r(\R)$ whose congruence action preserves the
affine slice and for which the projection intertwines this action with
rotation of the octagon. Write $\xcPSD^{C_8}(R_8)$ for the least size
in this model.

\begin{corollary}\label{cor:equivariant-octagon}
The regular octagon satisfies
\[
 \xcPSD(R_8)=4,\qquad \xcPSD^{C_8}(R_8)=5.
\]
\end{corollary}
\begin{proof}
Theorem~\ref{thm:octagons} gives the unrestricted value. The equivariant
construction for regular $2^k$-gons in \cite{FSP2017polygons} has size
$2k-1$, giving the upper bound five. We prove the matching lower bound
using a small instance of the Fourier separation method from that paper.

Suppose a size-four rotation-equivariant lift exists. The equivariant
factorization theorem of \cite{FSP2015equivariant} gives covariant
vertex factors. Averaging an inner product makes the representation
orthogonal; write $R$ for its generator. Since rotations act freely
and transitively on facets, one facet factor can also be rotated to
give covariant facet factors. The ranks are therefore constant on each
side, and Proposition~\ref{prop:rank-budget} forces one side to have
rank one. If the vertex factors are $R^jbb^{\mathsf T}R^{-j}$, a fixed
facet slack is a sum of squares of coordinate functions of $R^jb$.
If the facet factors are $R^iaa^{\mathsf T}R^{-i}$, the same slack is
$\|B_0^{1/2}R^{-j}a\|^2$. In either case it has an SOS certificate in
a rotation-invariant real function space of dimension at most four.

Index the vertices by $j\in\mathbb Z_8$, put $\delta=\pi/8$, and use
the facet slack
\[
 \ell(j)=\cos\delta-\cos(\pi j/4+\delta).
\]
Write $\chi_k(j)=e^{\pi i k j/4}$ and
$H_k=\operatorname{span}_{\R}\{\Re\chi_k,\Im\chi_k\}$.
The real function space decomposes into $H_0,H_1,H_2,H_3,H_4$;
the first and last have dimension one and the others dimension two.
A certifying invariant space of dimension at most three would give a
size-three lift by Proposition~\ref{prop:sos-rank}, which is impossible.
Thus its Fourier support is either
\[
 K=\{\pm a,\pm b\}\quad(1\le a<b\le3),
 \qquad\text{or}\qquad K=\{0,4,\pm a\}\quad(1\le a\le3).
\]
If all frequencies have the same parity, their differences are even,
so sums of squared moduli cannot have the nonzero frequency-one
coefficient of $\ell$. Multiplying functions by $\chi_4(j)=(-1)^j$
preserves their squares and shifts $K$ by four. The remaining supports
therefore reduce to $\{\pm1,\pm2\}$ and $\{0,4,\pm1\}$.

For $\gamma=1,2$, define a complex-linear functional by
\[
 L_\gamma(\chi_k)=
 \begin{cases}
 e^{-i\widetilde k\delta},&
 k\equiv\widetilde k\pmod8,\quad |\widetilde k|\le\gamma,\\
 0,&\text{otherwise}.
 \end{cases}
\]
For $K=\{\pm1,\pm2\}$ take $\gamma=1$ and the two clusters
$\{-2,-1\}$, $\{1,2\}$. For $K=\{0,4,\pm1\}$ take $\gamma=2$
and the clusters $\{-1,0,1\}$, $\{4\}$. In each case the matrix
$(L_\gamma(\chi_{k-l}))_{k,l\in K}$ is zero between clusters and has
rank-one Hermitian PSD blocks within them. Hence
$L_\gamma(|f|^2)\ge0$ for every function supported on $K$, whereas
$L_\gamma(\ell)=\cos\delta-1<0$. This contradiction excludes size four.
\end{proof}

\section{Open problems}\label{sec:open}
The universal count strengthens the quantitative bounds while leaving
the following unrestricted questions open. We use the labels
P1--P5 to distinguish the different targets.
\begin{enumerate}[label=\textbf{P\arabic*.},leftmargin=2.5em,itemsep=5pt]
\item \emph{Permutahedron.} Determine the order of $\xcPSD(\Pi_n)$.
The bounds are $n+\log_3n-O(1)$ and $O(n\log n)$. Can unrestricted
PSD lifts have size $o(n\log n)$?
\item \emph{Universal vertex count.} Does $\xcPSD(P)=\Omega(\log N)$
hold uniformly for polytopes with $N$ vertices, as asked in
\cite[Section 5]{FawziSafey2018}? This would imply
$\xcPSD(\Pi_n)=\Theta(n\log n)$. Equation~\eqref{eq:universal-count}
still has a quadratic exponent in $r$ when $d$ is fixed.

Observe that the sharper inequality $N\le2^{\xcPSD(P)-1}$ is false. Indeed,
$Q=\{x\in[-1,1]^7:\sum_i x_i=0\}$ has $140$ vertices and
$\xcPSD(Q)\le8$. At a vertex of $Q$, at most one coordinate can lie
in $(-1,1)$, since two such coordinates could be perturbed in opposite
directions. Seven signs cannot sum to zero; the other six signs have
an even sum, so the remaining coordinate must be zero, with three signs
of each kind. Conversely, fixing those six signs determines a unique
point of $Q$, hence a vertex. Thus there are $7\binom63=140>2^7$
vertices. The cube is two-level and has PSD rank eight \cite{GRT2013};
its section inherits that lift by imposing the additional affine equation.
\item \emph{Rank-four polygons.} Can a polygon of PSD rank four have
more than eight vertices? Theorem~\ref{thm:octagons} supplies a family
attaining eight.
Theorem~\ref{cor:rank-four} leaves only nine through twelve vertices
to exclude.
\item \emph{Polygon growth.} Does every convex $N$-gon have PSD rank
$\Omega(\log N)$, or even at least $1+\lceil\log_2N\rceil$?
For regular polygons, Vandaele--Glineur--Gillis
\cite[Conjecture 1]{VGG2018} conjecture equality with the latter
expression. Logarithmic lower bounds for equivariant lifts
\cite{FSP2017polygons} impose an additional restriction and do not
settle this question.
\item \emph{SOS versus unrestricted lifts.} Can unbounded vertex-factor
ranks give asymptotically smaller lifts for permutahedra or regular
polygons? Corollary~\ref{cor:bounded-vertex-ranks} rules out an
improvement when these ranks stay bounded.
Corollary~\ref{cor:rank-distribution} quantifies the ranks required
at almost every vertex in hypothetical smaller factorizations.
\end{enumerate}
A finite target for P1 is $\xcPSD(\Pi_5)\ge7$. Excluding size-six
factorizations would, by the face $\Pi_5\subset\Pi_n$, give
$\xcPSD(\Pi_n)\ge n+2$ for every $n\ge5$, extending the range
currently obtained from vertex counting.

\subsection{Scope of the comparison methods}\label{sec:barriers}
The following support-based restriction is due to Lee--Theis
\cite{LeeTheis2012}; we include its short matrix proof.
\begin{proposition}\label{prop:support}
Let $b$ be a function of zero patterns such that
$b(\operatorname{supp}M)\le\rkpsd(M)$ for every nonnegative matrix $M$.
Then $b(\operatorname{supp}M)\le\rank M$ for every such $M$.
\end{proposition}
\begin{proof}
The matrix $M\circ M$ has the same support as $M$, and
$\rkpsd(M\circ M)\le\rank M$ by Lemma~\ref{lem:basic}.
\end{proof}
For slack matrices the resulting bound is at most $\dim P+1$.
This does not limit every argument using the face lattice: the matrix
$M\circ M$ need not itself be a polytope slack matrix, and the
hexagonal-face lower bound uses that additional structure.
Likewise, sharp algebraic-degree bounds for general convex bodies
\cite{FawziSafey2018} do not preclude stronger bounds for polytopes.

\begin{remark}
The parity polytope
$\Par_n=\conv\{x\in\{\pm1\}^n:\prod_i x_i=1\}$ has linear lifts
of size $O(n)$ \cite{CarrKonjevod2005,KaibelPashkovich2011}, whereas
for $n\ge8$ its PSD lifts equivariant under evenly signed permutations
have size at least $\binom n{\lceil n/4\rceil}$
\cite[Theorem 2]{FSP2015equivariant}. Thus a symmetry restriction can
change the order of the problem. Similarly, a prescribed SOS space
is an additional restriction.
For unrestricted lifts, Proposition~\ref{prop:sos-rank} gives, when
$\xcPSD(\Pi_n)=n+t$,
\[
 \sosc(\Pi_n)\le(t+1)\left(n+\frac t2\right).
\]
This explains how the SOS lower bound yields an additive logarithmic
term above dimension, without determining unrestricted PSD order.
\end{remark}
\bibliographystyle{plain}
\bibliography{refs}
\end{document}